\documentclass[journal,twoside,web]{ieeecolor}
\usepackage{generic}
\usepackage{cite}
\usepackage{amsmath,amssymb,amsfonts}
\usepackage{graphicx}
\usepackage{textcomp}

\usepackage[pdftex,dvipsnames]{xcolor}                
\usepackage{tikz}
\usepackage{array}
\newcolumntype{P}[1]{>{\centering\arraybackslash}p{#1}} 
\usepackage{textcomp}
\usepackage{stfloats}
\usepackage{verbatim}
\usepackage{dirtytalk}
\usepackage{multirow}
\usepackage{rotating}
\usepackage{booktabs}
\usepackage{tikz}
\usepackage{threeparttable} 
\usepackage{transparent}

\usepackage{graphicx}
\usepackage{amsmath,amsfonts}
\usepackage{bm}
\makeatletter
\let\NAT@parse\undefined
\makeatother
\usepackage{hyperref}
\usepackage{url}
\usepackage{subfig}
                    
\usepackage{subcaption}
\usepackage{algorithm} 
\usepackage{algpseudocode} 
\usepackage{tikz-network}
\usetikzlibrary{arrows,automata,positioning}

\usepackage{tikz}
\usetikzlibrary{trees}
\usetikzlibrary{chains, fit}
\usepackage{listings}
\usepackage{xcolor,colortbl}
\usetikzlibrary{shapes,arrows}
\usetikzlibrary{shadings}
\usetikzlibrary{automata,topaths}
\usetikzlibrary{shapes}
\usetikzlibrary{shapes.geometric, arrows}
\usepackage{tikz-3dplot}
\usetikzlibrary{matrix}
\usepackage{comment}

\usepackage{pgfplots}
\usetikzlibrary{pgfplots.groupplots}
\usepackage{enumerate}

\newcommand{\R}{{\rm I\!R}}

\usetikzlibrary{quotes}

\def\BibTeX{{\rm B\kern-.05em{\sc i\kern-.025em b}\kern-.08em
    T\kern-.1667em\lower.7ex\hbox{E}\kern-.125emX}}
\usepackage{balance}
  
\newcommand{\boundellipse}[3]
{(#1) ellipse (#2 and #3)
}

\newcommand{\cf}{{\emph{cf.~}}}
\newcommand{\tran}{^{\mbox{\scriptsize T}}}
\usepackage{pifont}
\newtheorem{theorem}{Theorem}
\newtheorem{lemma}[theorem]{Lemma}
\newtheorem{assump}{Assumption}

\numberwithin{assump2}{section} 

\newtheorem{remark}{Remark}

\newtheorem{prop}{Proposition}

\numberwithin{prop2}{section} 

\newtheorem{definition}{Definition}

\newtheorem{corollary}{Corollary}
\newtheorem{example}{Example}

\makeatletter
\newcommand{\nosemic}{\renewcommand{\@endalgocfline}{\relax}}
\newcommand{\dosemic}{\renewcommand{\@endalgocfline}{\algocf@endline}}
\let\oldnl\nl
\newcommand{\nonl}{\renewcommand{\nl}{\let\nl\oldnl}}
\makeatother
\usepackage{amssymb}

\def\BibTeX{{\rm B\kern-.05em{\sc i\kern-.025em b}\kern-.08em
    T\kern-.1667em\lower.7ex\hbox{E}\kern-.125emX}}

\begin{document}
\title{Effective Dual Subgradient Methods Through Sequential Subgradient-Aligned Rays}
\author{Hansi Abeynanda, Chathuranga Weeraddana, and Carlo Fischione, \IEEEmembership{Fellow, IEEE}
\thanks{H. Abeynanda is with the School of Electrical Engineering and Computer Science, KTH Royal Institute of Technology, 11428 Stockholm, Sweden (e-mail: hkab@kth.se).}
\thanks{C. Weeraddana is with the Centre for Wireless Communication,
University of Oulu, Finland (e-mail: chathuranga.weeraddana@oulu.fi).}
\thanks{C. Fischione is with the School of Electrical Engineering and Computer Science and the Digital Futures, KTH Royal Institute of Technology, 11428 Stockholm, Sweden (e-mail: carlofi@kth.se).}}

\maketitle

\begin{abstract}
In this paper, we propose an efficient subgradient-based framework, referred to as the \emph{sequential subgradient-aligned rays} (SSAR) method, for improving the convergence performance of the dual subgradient method. In particular, the proposed method identifies and exploits a finite sequence of rays, when available, as the dual subgradient method evolves, along each of which a subgradient of the negative dual function remains constant and aligns with the ray. While the dual variable iterates remain on each identified ray, a constant stepsize associated with the ray is employed, with the associated constant stepsizes decreasing from one ray to the next. This approach yields constant decrements in the negative dual function values while the dual variable iterates remain on the associated rays, enabling the dual variable iterates to progress more quickly toward a dual optimal solution. We then show that the proposed SSAR method can be leveraged to improve the performance of the dual decomposition method for minimizing the average of local loss functions associated with multiple agents that share a common decision variable, a widely used problem formulation in many large-scale real-world applications. In particular, the SSAR method improves the convergence performance of the corresponding dual decomposition method while also reducing its communication cost. Numerical experiments on a least squares regression problem and a logistic regression problem using real datasets are conducted to evaluate the performance of the SSAR method.

\end{abstract}


\section{Introduction}\label{sec:Introduction}
Optimization problems arising in networked systems often involve multiple spatially distributed agents that collaboratively perform learning, estimation, resource allocation, or decision-making tasks using locally available information~\cite{Nedic-Distributed-Optimization-for-Control-2018,notarnicola2019constraint}. Such problems arise, for example, in sensor networks, communication networks, and multi-agent control systems, where individual agents may represent sensors, communication nodes, or dynamical subsystems that exchange information over a communication network~\cite{jiang2022distributed}. In many of these applications, the resulting optimization problems are inherently large-scale, involving large datasets and high-dimensional decision variables, while the information required to formulate the overall problem is distributed among the agents. Consequently, there has been significant interest in developing distributed optimization methods that enable agents to collaboratively compute an optimal solution through local computations and limited communication without requiring the direct exchange of their local data.

The effectiveness of distributed optimization methods in networked systems depends critically on their convergence performance and communication efficiency. Motivated by these considerations, we propose an efficient subgradient-based framework to improve both the convergence performance and communication efficiency of the dual decomposition method for solving the following distributed optimization problem, which arises in a broad range of applications, including signal processing, machine learning, wireless communications, and control systems~\cite{Cao-2012-overview,di_2020,Nedic-Distributed-Optimization-for-Control-2018,Yang-survey-of-distributed-optimization-2019,Carlo-Machine-Learning}:
\begin{equation} \label{eq:main-problem}
\begin{array}{ll}
\mbox{minimize} & f_0(\mathbf{z}) = \frac{1}{m}\sum_{i=1}^{m} f_i(\mathbf{z}) \\
\mbox{subject to} & \mathbf{z} \in \mathcal{Y},
\end{array}
\end{equation}
where $\mathbf{z} \in \R^n$ denotes the optimization variable, $\mathcal{S} = \{1,\ldots,m\}$ denotes the set of agents, and $\mathcal{Y} \subseteq \R^n$ is the constraint set. For each $i \in \mathcal{S}$, the function $f_i : \R^n \to \R$ is known only to agent $i$, while the optimization variable $\mathbf{z}$ is shared among all agents. In practical applications, $f_i$ may represent a local loss, cost, or utility function associated with agent $i$, e.g., a loss over its local dataset, an energy consumption cost in a sensor network, or a local control cost associated with a subsystem in a networked control system. 

To compute an optimal solution to problem~\eqref{eq:main-problem}, the agents in the distributed system are required to collaborate through local computations and information exchange. In this regard, numerous primal distributed algorithms~\cite{nedic2009distributed,shi2015extra,qu2017harnessing,pu2021distributed}, dual decomposition methods~\cite{lin2021asynchronous,Sindri_2020,Necoara-inexact-dual-decomposition-2014,Yifan-Asynchrony-and-Inexactness-2023,abeynanda2023primal}, and augmented Lagrangian methods (ALMs)~\cite{condat_2013_splitting_methods,Boyd-Parikh-Chu-Peleato-Eckstein-2010,He_2021,Bai_2024} have been proposed in the literature. Among these approaches, dual decomposition combined with the subgradient method~\cite{lin2021asynchronous,Sindri_2020,Necoara-inexact-dual-decomposition-2014,Yifan-Asynchrony-and-Inexactness-2023,abeynanda2023primal} is particularly attractive due to its conceptual simplicity and ability to exploit problem separability, \cf~\cite{Boyd-EE364b-PrimDualDecomp-07}. Therefore, in this paper, we adopt a dual decomposition framework combined with the subgradient method to solve problem~\eqref{eq:main-problem}. In particular, we consider a coordinator-assisted architecture in which the agents perform local computations and communicate with a central server that coordinates the dual variable updates.

Then, it is important to investigate how the convergence performance and
communication efficiency of dual decomposition methods can be improved
when solving problem~\eqref{eq:main-problem}. One of the key factors
affecting the convergence performance of the corresponding dual
decomposition is the choice of the stepsize in the dual
subgradient method. Thus, various stepsize rules~\cite{Polyak_Intro_Opt,nesterov2018lectures,Thinh_2018,liu2024distributed,Damek_2018,zhu_2019,cutler2023stochastic,ge_2019,wang_2021,zagoruyko2016wide,huang2017densely,Yura_2020,abeynanda2026characteristics} proposed in the literature can be employed.

Traditional stepsize rules in subgradient methods include
constant stepsizes (\cf~\cite[\S~1.4]{Polyak_Intro_Opt},
\cite[\S~2.1.5]{nesterov2018lectures}), polynomially decaying stepsizes
(\cf~\cite{Thinh_2018,liu2024distributed}), and geometrically decaying
stepsizes (\cf~\cite{Damek_2018,zhu_2019}). The step-decay stepsize rule~\cite{cutler2023stochastic,ge_2019,wang_2021,zagoruyko2016wide,huang2017densely} is another commonly used stepsize strategy in the literature. This strategy divides the iterations into stages and uses a constant stepsize within each stage while reducing the stepsize at each subsequent stage. In particular, step-decay stepsizes have been widely used for training deep neural networks~\cite{zagoruyko2016wide,huang2017densely}. The stepsize rule proposed in~\cite{Yura_2020}
is iterate-dependent and adaptive to the local smoothness of the objective
function. 

The stepsize rule proposed in~\cite{abeynanda2026characteristics} is specifically designed for dual subgradient methods by exploiting a characteristic of the dual function referred to as \emph{fixed gradient over rays} (FGOR), or \emph{fixed subgradient over rays} (FSGOR) when the dual function is nondifferentiable, \cf~\S~\ref{subsec:FSGOR-Method}. In particular, the FSGOR method~\cite{abeynanda2026characteristics} initializes the dual subgradient method from a ray along which a subgradient of the negative dual function remains constant and aligns with the ray, and employs a constant stepsize while the dual variable iterates remain on this ray. Once the iterates leave the ray, the method switches to an existing stepsize rule. This approach has been shown to considerably improve the performance of dual decomposition methods applied with existing stepsize rules~\cite{Polyak_Intro_Opt,nesterov2018lectures,Thinh_2018,liu2024distributed,Damek_2018,zhu_2019,cutler2023stochastic,ge_2019,wang_2021,zagoruyko2016wide,huang2017densely,Yura_2020}, while also reducing communication overhead.

However, FSGOR method~\cite{abeynanda2026characteristics} exploits only the initial ray and does not investigate whether subsequent dual variable iterates exhibit the same structural property. This raises the possibility of further exploiting this structure beyond the initial ray. Motivated by this observation (\cf~\S~\ref{subsec:motivation-to-AdFSGOR}), we propose an extension of FSGOR, referred to as the \emph{sequential subgradient-aligned rays} (SSAR) method. The main contributions of this paper are summarized as follows.
\begin{enumerate}
\item \emph{Sequential subgradient-aligned rays (SSAR) method}: We propose SSAR to improve the convergence of the dual subgradient method by exploiting a finite sequence of identifiable rays along which the negative dual function has a constant subgradient aligned with each ray. Using constant stepsizes along individual rays that decrease across successive rays, SSAR achieves constant reductions in the negative dual function while iterates remain on these rays, thereby enabling faster progress toward a dual optimal solution, \cf~\S~\ref{subsec:Development-of-the-AdFSGOR-Method}, Algorithm~\ref{Alg:AdFSGOR-algorithm}.
\item \emph{SSAR-assisted dual decomposition method:} We show how the SSAR method is integrated into the dual decomposition framework for solving~\eqref{eq:main-problem}, \cf~\S~\ref{subsec:AdFSGOR-Assisted-Dual-Decomposition-Method}, Algorithm~\ref{Alg:AdFSGOR-assisted_dual-decomposition-method}.
\item \emph{Computation of restricted-FSGOR rays:} We discuss how to determine the initial restricted-FSGOR rays (\cf~Definition~\ref{def:RFSGOR-ray}) associated with the consensus problem~\eqref{eq:consensus-problem}, which are used to initialize the SSAR-assisted dual decomposition algorithm, \cf~\S~\ref{Sec:computation-of-RFSGOR-rays}.
\item \emph{Communication efficiency of SSAR:} We show that SSAR reduces the communication cost of the
standard dual decomposition method for solving~\eqref{eq:main-problem} and further enhances
the communication efficiency of FSGOR~\cite{abeynanda2026characteristics}, \cf~\S~\ref{Sec:Communication-Efficiency-of-AdFSGOR}.
\item Numerical results demonstrate that SSAR significantly improves the convergence performance and communication efficiency of standard dual decomposition and outperforms FSGOR~\cite{abeynanda2026characteristics} in both aspects, \cf~\S~\ref{sec:numerical-results}.
\end{enumerate}

The rest of this paper is organized as follows. In~\S~\ref{Sec:Preliminaries}, we discuss preliminaries related to our study. The proposed SSAR method and its associated technical details are presented in~\S~\ref{Sec:AdFSGOR-method}. Next, in~\S~\ref{sec:AdFSGOR-Assisted-Dual-Decomposition-Method-and-computation-of-RFSGOR-rays}, we present the SSAR-assisted dual decomposition algorithm and discuss how to compute initial restricted-FSGOR rays associated with the consensus problem~\eqref{eq:consensus-problem}. The communication efficiency of SSAR is discussed in~\S~\ref{Sec:Communication-Efficiency-of-AdFSGOR}. Numerical results are presented in~\S~\ref{sec:numerical-results}. Finally, conclusions are provided, followed by the
appendices.

\emph{Notations:} The set of real numbers, set of positive real numbers,
set of extended real numbers, set of real $n$-vectors, set of real
$m\times n$ matrices, set of positive integers, and set of nonnegative
integers are denoted by $\R$, $\R^+$, $\overline{\R}$, $\R^n$,
$\R^{m\times n}$, $\mathbb{Z}_+$, and $\mathbb{Z}^0_+$, respectively.
The boundary, interior, and convex hull of a set $\mathcal{X}$
are denoted by $\texttt{bnd}~\mathcal{X}$, $\texttt{int}~\mathcal{X}$, and $\texttt{con}~\mathcal{X}$, respectively. The domain of a function $f:\R^n\to\overline{\R}$ is denoted
by $\texttt{dom}~f$. The set of all
subgradients of a function $f$ at a point $\mathbf{x}\in\R^n$ is denoted
by $\partial f(\mathbf{x})$. The range of a matrix $\mathbf{X}$ is denoted
by $R(\mathbf{X})$. The $n\times n$ identity matrix, $n\times n$ zero
matrix, $n$-vector with all entries equal to one, and $n$-vector with all
entries equal to zero are denoted by $\mathbf{I}_n$, $\mathbf{O}_n$,
$\mathbf{1}_n$, and $\mathbf{0}_n$, respectively.

\section{Preliminaries}\label{Sec:Preliminaries}
We first discuss the FSGOR method~\cite{abeynanda2026characteristics}, which provides the basis for motivating the proposed SSAR method. Next, we present the dual decomposition method applied to problem~\eqref{eq:main-problem}.

\subsection{The FSGOR Method~\cite{abeynanda2026characteristics}}\label{subsec:FSGOR-Method}
Consider the following generic optimization problem:
\begin{equation}\label{eq:FSGOR-main-problem} 
\begin{array}{ll}
\mbox{minimize} & h_0(\mathbf{x})\\
\mbox{subject to} & \mathbf{x}\in\mathcal{X}\\
& \mathbf{A}\mathbf{x} = \mathbf{b},
\end{array}
\end{equation}
where $\mathbf{x}\in\R^n$, $h_0:\R^{n}\rightarrow \R$, 
$\mathcal{X}\subseteq\R^n$, $\mathbf{A}\in\R^{p\times n}$ with 
$p<n$ and $\texttt{rank}~\mathbf{A}=p$, and $\mathbf{b}\in\R^{p}$.
Problem~\eqref{eq:FSGOR-main-problem} can be equivalently reformulated as
\begin{equation} \label{eq:FSGOR-main-problem_with-indicator-function}
\begin{array}{ll}
\mbox{minimize} & h(\mathbf{x})=h_0(\mathbf{x}) + \delta_\mathcal{X}(\mathbf{x})\\
\mbox{subject to} & \mathbf{A}\mathbf{x} = \mathbf{b},
\end{array}
\end{equation}
where $\delta_{\mathcal{X}}$ denotes the indicator function of the set
$\mathcal{X}$. Next, we make the following assumptions.
\begin{assump}\label{Assumption:about-h0-and-X}
The function $h_0$ is convex, lower semicontinuous, and proper, and $\mathcal{X}$ is convex and compact.   
\end{assump}
\begin{assump}[cf.~{\cite[Assumption~2]{abeynanda2026characteristics}}]\label{Assumption:about-subdifferential-of-h-at-boundary-ponts} $\forall~\mathbf{\bar x}\in\texttt{bnd}~\mathcal{X}$, there exists a sequence $\{\mathbf{x}_k\}_{k=1}^{\infty}$, $\mathbf{x}_k\in\mathcal{X}$, with $\lim_{k\to\infty}\mathbf{x}_k=\mathbf{\bar x}$ such that $\lim\sup_{k\rightarrow\infty} \ \|\boldsymbol{\nu}_k\|_2 \neq \infty  \ \mbox{with} \ \boldsymbol{\nu}_k\in\partial h(\mathbf{x}_k)$.
\end{assump}

The dual function $g:\R^p\to\overline{\R}$ associated with
problem~\eqref{eq:FSGOR-main-problem_with-indicator-function} is given by $g(\boldsymbol{\lambda})
=
\inf_{\mathbf{x}\in\R^n}
\left(
h(\mathbf{x})
-\boldsymbol{\lambda}\tran(\mathbf{A}\mathbf{x}-\mathbf{b})
\right)$, where $\boldsymbol{\lambda}\in\R^p$ denotes the Lagrange multiplier
associated with the equality constraint in
\eqref{eq:FSGOR-main-problem_with-indicator-function}. Then, the dual problem associated with problem~\eqref{eq:FSGOR-main-problem_with-indicator-function} is given by
\begin{equation}\label{eq:FSGOR-dual-problem}
\underset{\boldsymbol{\lambda}\in\R^p}{\text{maximize}}
\quad
g(\boldsymbol{\lambda}).  
\end{equation}
The subgradient method for solving~\eqref{eq:FSGOR-dual-problem} is given by
\begin{equation}
\label{eq:FSGOR-subgradient-method}
\boldsymbol{\lambda}^{(k+1)}
=
\boldsymbol{\lambda}^{(k)}
-
\gamma_k
\mathbf{d}^{(k)},
\end{equation}
where $\gamma_k$ denotes the stepsize, $\mathbf{d}^{(k)}\in\partial(-g)(\boldsymbol{\lambda}^{(k)})$, and $k$ denotes the iteration index. In particular, we have that $\mathbf{d}^{(k)}=\mathbf{A}\mathbf{x}^{(k)}{-}\mathbf{b}$ is a subgradient of $-g$ at $\boldsymbol{\lambda}^{(k)}$,
where
\begin{equation}
\label{eq:minimizer-of-the-Lagrangian}
\mathbf{x}^{(k)}
\in
\arg\min_{\mathbf{x}\in\mathcal{X}}(h_0(\mathbf{x})-\boldsymbol{\lambda}^{{(k)}\text{T}}\mathbf{A}\mathbf{x}+\boldsymbol{\lambda}^{{(k)}\text{T}}\mathbf{b}).
\end{equation}

The FSGOR method~\cite{abeynanda2026characteristics} initializes
the dual subgradient method~\eqref{eq:FSGOR-subgradient-method} from a ray
in $\texttt{dom}~g$ along which a subgradient of the negative dual function
remains constant and aligns with the ray. Such a ray is referred to as a
restricted-FSGOR (RFSGOR) ray. Then, as long as the
dual variable iterate $\boldsymbol{\lambda}^{(k)}$ lies on the ray, a
constant stepsize is used, and an existing stepsize rule is employed when
$\boldsymbol{\lambda}^{(k)}$ leaves the ray. For clarity, we define an RFSGOR ray below. 
\begin{definition}[RFSGOR ray]\label{def:RFSGOR-ray}
A ray $\{\boldsymbol{\lambda}+\alpha\boldsymbol{\mu}\mid\alpha\geq0\}$ in $\texttt{dom}~g$ along which a subgradient of $-g$ remains constant and aligns with the ray is referred to as an RFSGOR ray.
\end{definition}

To apply the stepsize rule in~\cite{abeynanda2026characteristics}, the computation
of an RFSGOR ray is crucial. In this regard, for a given
$\mathbf{\bar x}\in\texttt{bnd}~\mathcal{X}$ associated with
problem~\eqref{eq:FSGOR-main-problem}, the following feasibility problem is used,
\cf~\cite[eq.~(26)]{abeynanda2026characteristics}:
\begin{equation} \label{eq:RFSGOR-feasibility-problem-original}
\begin{array}{ll}
\underset{(\boldsymbol{\lambda},\,\boldsymbol{\mu},\,\boldsymbol{\nu},\,\boldsymbol{\eta})}{\text{minimize}} & 0\\
\mbox{subject to} & \mathbf{A}\mathbf{\bar x}- \mathbf{b} = \beta(\mathbf{A}\mathbf{A}\tran)^{-1}\mathbf{A}{\boldsymbol\eta}\\
& {\boldsymbol\nu} \in \partial h_0(\mathbf{\bar x}) + \bar{N}_{\mathcal{X}}(\mathbf{\bar x})\\
& {\boldsymbol\eta} \in \bar{N}_{\mathcal{X}}(\mathbf{\bar x})\\
& \mathbf{A}\tran \boldsymbol{\lambda} =  {\boldsymbol\nu} \\
& \mathbf{A}\tran \boldsymbol{\mu} =  {\boldsymbol\eta},
\end{array}
\end{equation}
  where $\bar{N}_{\mathcal{X}}(\mathbf{\bar x})
  =\texttt{int}~N_{\mathcal{X}}(\mathbf{\bar x})$.
  If problem~\eqref{eq:RFSGOR-feasibility-problem-original} is feasible with
  solution $(\boldsymbol{\hat\lambda},\boldsymbol{\hat\mu},
  \boldsymbol{\hat\nu},\boldsymbol{\hat\eta})$, then the RFSGOR ray
  associated with problem~\eqref{eq:FSGOR-main-problem} is given by (\cf~\cite[eq.~(27)]{abeynanda2026characteristics}):
  \begin{equation}\label{eq:RFSGOR-ray-original}
  \hat{\mathcal{R}}(\mathbf{\bar x})
  =
  \{\boldsymbol{\hat\lambda}
  +\alpha\boldsymbol{\hat\mu}\mid\alpha\geq0\}.
  \end{equation}

When the entities in the optimization system do not have access to the
global objective function $h_0$ (e.g., if
problem~\eqref{eq:FSGOR-main-problem} takes the form~\eqref{eq:consensus-problem},
neither the central server nor the agents have access to $f$),
the following relaxed problem is used instead of
\eqref{eq:RFSGOR-feasibility-problem-original} to compute an RFSGOR ray,
\cf~\cite[eq.~(28)]{abeynanda2026characteristics}:
\begin{equation} \label{eq:RFSGOR-feasibility-problem-relaxed}
\begin{array}{ll}
\underset{(\boldsymbol{\mu},\boldsymbol{\eta})}{\text{minimize}}
& 0\\
\mbox{subject to}
& \mathbf{A}\mathbf{\bar x}-\mathbf{b}
= \beta(\mathbf{A}\mathbf{A}\tran)^{-1}\mathbf{A}\boldsymbol{\eta}\\
& \boldsymbol{\eta}\in\bar{N}_{\mathcal{X}}(\mathbf{\bar x})\\
& \mathbf{A}\tran\boldsymbol{\mu}=\boldsymbol{\eta}.
\end{array}
\end{equation}

If problem~\eqref{eq:RFSGOR-feasibility-problem-relaxed} is feasible, then
problem~\eqref{eq:RFSGOR-feasibility-problem-original} is also feasible, as established by the following result. Consequently, an RFSGOR ray
$\hat{\mathcal{R}}(\mathbf{\bar x})$ (\cf~\eqref{eq:RFSGOR-ray-original}) can be computed,
\cf~\cite[eq.~(29)]{abeynanda2026characteristics}.
\begin{lemma}\label{lemma:feasibility-of-relaxed-problem-implies-feasibility-of-original-problem}
Suppose problem~\eqref{eq:RFSGOR-feasibility-problem-relaxed} is
feasible for some $\mathbf{\bar x}\in\texttt{bnd}~\mathcal{X}$, and let
$(\boldsymbol{\tilde\mu},\boldsymbol{\tilde\eta})$ denote a corresponding
solution of~\eqref{eq:RFSGOR-feasibility-problem-relaxed}. Then, for
sufficiently large $c>0$, problem~\eqref{eq:RFSGOR-feasibility-problem-original}
is also feasible with $(\boldsymbol{\hat\lambda},\boldsymbol{\hat\mu},
\boldsymbol{\hat\nu},\boldsymbol{\hat\eta})=(c\boldsymbol{\tilde\mu},
\boldsymbol{\tilde\mu},
c\boldsymbol{\tilde\eta},
\boldsymbol{\tilde\eta})$.
\end{lemma}

\begin{proof}
The proof is provided in Appendix~\ref{Appendix:proof-feasibility-of-relaxed-problem-implies-feasibility-of-original-problem}.
\end{proof}
The feasibility of problem~\eqref{eq:RFSGOR-feasibility-problem-relaxed}
can be checked as follows.
\begin{remark}\label{Remark:feasibility-of-relaxed-problem}
By combining the first and third constraints in
problem~\eqref{eq:RFSGOR-feasibility-problem-relaxed}, we obtain
$\boldsymbol{\mu}=(\mathbf{A}\mathbf{\bar x}-\mathbf{b})/\beta$ and
$\boldsymbol{\eta}=\mathbf{A}\tran(\mathbf{A}\mathbf{\bar x}-\mathbf{b})/\beta$.
Thus, it suffices to check whether the second constraint is satisfied
to determine the feasibility of
problem~\eqref{eq:RFSGOR-feasibility-problem-relaxed}. Thus,
$(\boldsymbol{\mu},\boldsymbol{\eta})$ is feasible for
\eqref{eq:RFSGOR-feasibility-problem-relaxed} if and only if
$\boldsymbol{\eta}\in\texttt{int}~N_{\mathcal{X}}(\mathbf{\bar x})$.
\end{remark}

Note that for any nonempty convex set $\mathcal{C}\subseteq\R^n$, we have
$\texttt{int}~\mathcal{C}\neq\emptyset$ if and only if
$\texttt{aff}(\mathcal{C})=\R^n$. Applying this result to
$N_{\mathcal{X}}(\mathbf{\bar x})$, we obtain
$\texttt{int}~N_{\mathcal{X}}(\mathbf{\bar x})\neq\emptyset$ if and only if
$\texttt{aff}(N_{\mathcal{X}}(\mathbf{\bar x}))=\R^n$. Therefore, when checking the feasibility of problem~\eqref{eq:RFSGOR-feasibility-problem-relaxed} for computing an RFSGOR ray, we choose $\mathbf{\bar x}$ as specified in the following remark.
\begin{remark}\label{remark:condition-for-choosing-x-bar}
When checking the feasibility of problem~\eqref{eq:RFSGOR-feasibility-problem-relaxed} for computing an RFSGOR ray $\hat{\mathcal{R}}(\mathbf{\bar x})$, we choose $\mathbf{\bar x}\in\texttt{bnd}~\mathcal{X}$ such that
$\texttt{aff}(N_{\mathcal{X}}(\mathbf{\bar x}))=\R^n$.
\end{remark}

\subsection{The Dual Decomposition Method for Solving~\eqref{eq:main-problem}}
\label{subsec:The-Standard-Dual-Decomposition-Method}

We first reformulate~\eqref{eq:main-problem} into the equivalent consensus form:
\begin{equation} \label{eq:consensus-problem}
\begin{array}{ll}
\mbox{minimize} & f(\mathbf{y})=\frac{1}{m}\sum_{i=1}^{m}f_i(\mathbf{y}_{i}) \\
\mbox{subject to} & \mathbf{y}_i\in\mathcal{Y},\quad i\in \mathcal{S},\\ 
& \mathbf{y}_i=\mathbf{y}_{i+1},\quad i\in\mathcal{S}\setminus \{m\},
\end{array}
\end{equation}
where $\mathbf{y}_i\in\R^n$, $i\in\mathcal{S}$, denote the local copies
of the shared variable $\mathbf{z}$ in~\eqref{eq:main-problem}, and
$\mathbf{y}
=
[\mathbf{y}_{1}\tran\ \ldots\ \mathbf{y}_{m}\tran]\tran$.
Note that the equality constraints
$\mathbf{y}_i=\mathbf{y}_{i+1}$, $i\in\mathcal{S}\setminus\{m\}$, in
\eqref{eq:consensus-problem} can be written compactly as
$\mathbf{C}\mathbf{y}=\mathbf{0}_{n(m-1)}$, where
\begin{equation}\label{eq:consenses-problem-matrix-C}
\mathbf{C}=\begin{bmatrix}
\mathbf{I}_n & -\mathbf{I}_{n} & \mathbf{O}_n & \cdots & \mathbf{O}_n\\  
\mathbf{O}_n & \mathbf{I}_n & -\mathbf{I}_{n} & \ddots & \vdots \\ %
\vdots & \ddots & \ddots & \ddots & \mathbf{O}_n\\ %
\mathbf{O}_n & \cdots & \mathbf{O}_n & \mathbf{I}_n & -\mathbf{I}_n
\end{bmatrix}\in\R^{n(m-1)\times nm}.
\end{equation}
The dual problem associated with~\eqref{eq:consensus-problem}
is given by
\begin{equation}
\label{eq:dual-problem-of-consensus-problem}
\underset{\boldsymbol{\lambda}\in\R^{n(m-1)}}{\text{maximize}}
\quad
\sum_{i=1}^{m}
\underbrace{
\inf_{\mathbf{y}_{i}\in\mathcal{Y}}
\Big[
\frac{1}{m}f_i(\mathbf{y}_{i})
+
(\boldsymbol{\lambda}_i-\boldsymbol{\lambda}_{i-1})\tran
\mathbf{y}_{i}
\Big]
}_{\text{local subproblem of agent } i},
\end{equation}
where $\boldsymbol{\lambda}
=
[\boldsymbol{\lambda}_1\tran\ \ldots\ \boldsymbol{\lambda}_{m-1}\tran]\tran$
and $\boldsymbol{\lambda}_0=\boldsymbol{\lambda}_m=\mathbf{0}_n$.
Note that the dual function (\cf~\eqref{eq:dual-problem-of-consensus-problem}) decomposes into $m$ local
subproblems. Thus, for a given $\boldsymbol{\lambda}$, each local
subproblem can be solved independently by the corresponding agent $i$. Moreover, the subgradient method~\eqref{eq:FSGOR-subgradient-method} is used to solve~\eqref{eq:dual-problem-of-consensus-problem}. The dual decomposition algorithm is then given in Algorithm~\ref{Alg:Dual-Decomposition-Algorithm}.

\begin{algorithm}[!t]
	\caption{Dual Decomposition for Solving~\eqref{eq:main-problem}} 
	\begin{algorithmic}[1]
	    \Require $\boldsymbol{\lambda}\in\R^{n(m-1)}$.
	    \State $k=0$, $\boldsymbol{\lambda}^{(0)}=\boldsymbol{\lambda}$. $\boldsymbol{\lambda}_0^{(r)}=\boldsymbol{\lambda}_m^{(r)}=\mathbf{0}_n, \ r\in\mathbb{Z}^0_+$.
		\Repeat 
            \State $\forall i$, agent $i$ computes:
            \begin{equation}\label{eq:solution-of-local-subproblem}
\mathbf{y}_i^{(k)}=\underset{\mathbf{y}_i\in\mathcal{Y}}{\arg\min}~\frac{1}{m}f_i(\mathbf{y}_i)+{\big(\boldsymbol{\lambda}^{(k)}_i{-}\boldsymbol{\lambda}^{(k)}_{i-1}\big)}\tran \mathbf{y}_i.
             \end{equation}
 \State $\forall i$, agent $i$ transmits $\mathbf{y}_i^{(k)}$ to CS. %
 
            \State CS computes: ${\mathbf{d}}^{(k)}{=}[(\mathbf{y}^{(k)}_2{-}{\mathbf{y}}^{(k)}_1)\tran{\ldots}(\mathbf{y}^{(k)}_{m}{-}\mathbf{y}^{(k)}_{m{-}1})\tran]\tran$. 
            
            \noindent \ \ \ \  CS computes: $\boldsymbol{\lambda}^{(k+1)} = \boldsymbol{\lambda}^{(k)}- \gamma_k\mathbf{d}^{(k)}$. 
             \State $\forall i$, CS broadcasts $\boldsymbol{\lambda}^{(k+1)}$ to agent $i$. %
            \State $k:=k+1$.
        \Until{a stopping criterion true}
	\end{algorithmic} 
	\label{Alg:Dual-Decomposition-Algorithm}
\end{algorithm}

\section{The SSAR Method}\label{Sec:AdFSGOR-method}
In this section, we first discuss the motivation for introducing the SSAR method. Next, we develop the SSAR algorithm for solving problem~\eqref{eq:FSGOR-main-problem}.

\subsection{Motivation to SSAR}\label{subsec:motivation-to-AdFSGOR}

The FSGOR method~\cite{abeynanda2026characteristics} initializes the dual subgradient method from an RFSGOR ray, say $\mathcal{R}_1$ (\cf~Definition~\ref{def:RFSGOR-ray}), in the domain of the dual function $g$. The FSGOR method then enables the dual subgradient method to evolve along $\mathcal{R}_1$ using a constant stepsize, say $\gamma_1$. Once $\boldsymbol{\lambda}^{(k)}$ leaves the ray, the method switches to an existing stepsize rule~\cite{Polyak_Intro_Opt,nesterov2018lectures,Thinh_2018,liu2024distributed,Damek_2018,zhu_2019,cutler2023stochastic,ge_2019,wang_2021,zagoruyko2016wide,huang2017densely,Yura_2020}. This strategy yields constant decrements in the negative dual function $-g$ as long as $\boldsymbol{\lambda}^{(k)}\in\mathcal{R}_1$, allowing the dual variable iterates to progress more quickly toward a dual optimal solution $\boldsymbol{\lambda}^\star$.

However, the FSGOR method~\cite{abeynanda2026characteristics} identifies and exploits only the initial RFSGOR ray $\mathcal{R}_1$. After $\boldsymbol{\lambda}^{(k)}$ exits $\mathcal{R}_1$, depending on the structure of the underlying optimization problem, it may lie on another RFSGOR ray $\mathcal{R}_2$. If such a ray can be identified, the dual subgradient method can again employ a constant stepsize $\gamma_2$ as long as $\boldsymbol{\lambda}^{(k)}\in\mathcal{R}_2$. This process can be repeated to successively identify and exploit a finite number of subsequent RFSGOR rays, when available, using a constant stepsize along each identified ray. When no subsequent RFSGOR ray is identified or the prescribed maximum number of rays is reached, the method proceeds with an existing stepsize rule. We refer to this approach as the \emph{sequential subgradient-aligned rays} (SSAR) method. By extending the exploitation of the RFSGOR structure beyond the initial ray, SSAR can further improve the convergence performance of the dual subgradient method. In the following, we provide an illustrative example to give insight into the proposed approach.

\begin{figure}[!t]
\centering
\subfloat[]{\includegraphics[width=0.5\linewidth]%
{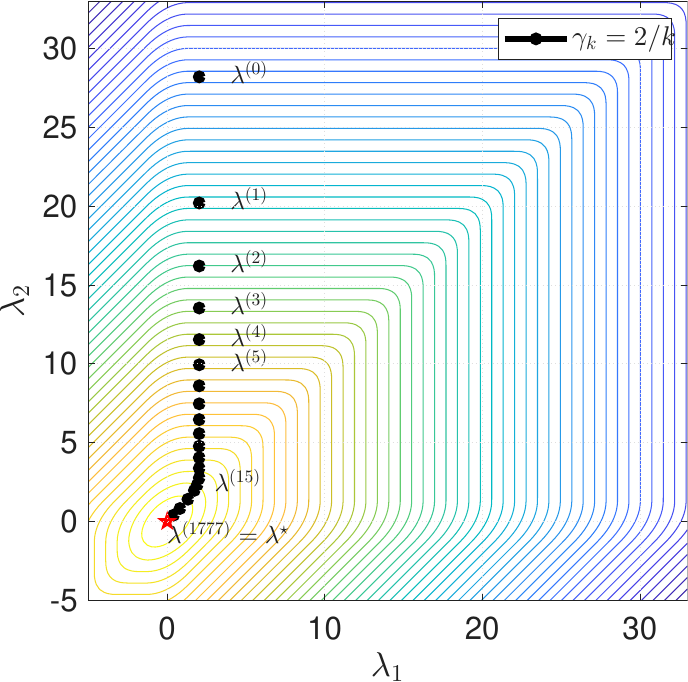}%
\label{subfig:path_of_lambda_diminishing_stepsize}}
\hfil
\subfloat[]{\includegraphics[width=0.5\linewidth]
{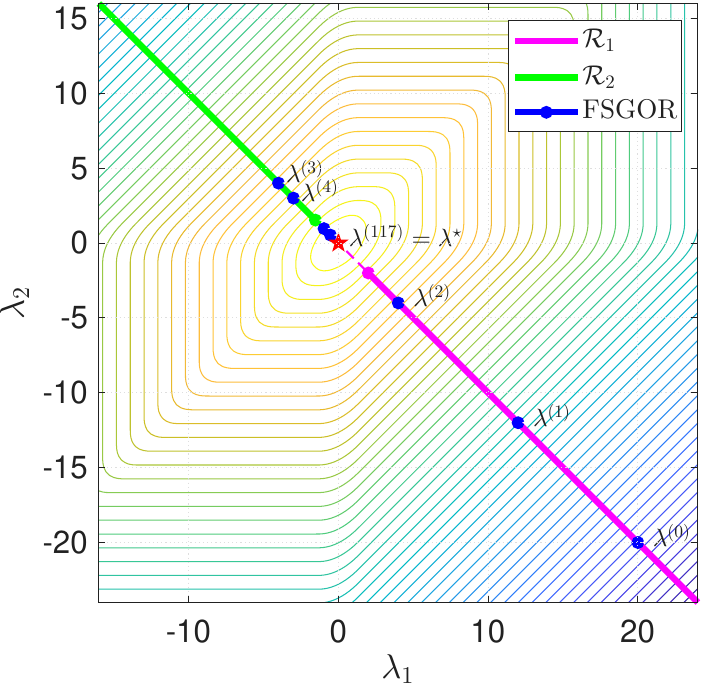}%
\label{subfig:path_of_lambda_FSGOR_initial_stepsize_2.eps}}
\\
\subfloat[]{\includegraphics[width=0.5\linewidth]%
{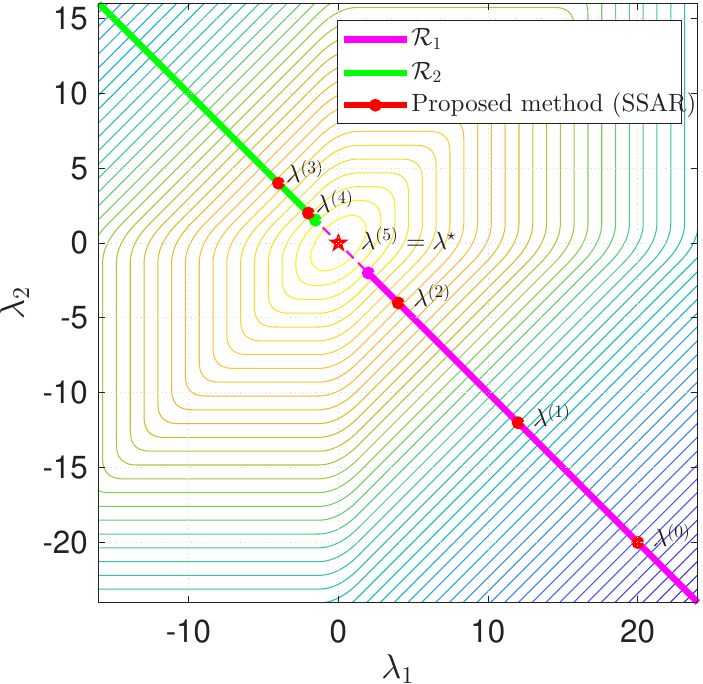}%
\label{subfig:path_of_lambda_AdFSGOR_initial_stepsize_2.eps}}
\hfil
\subfloat[]{\includegraphics[width=0.5\linewidth]%
{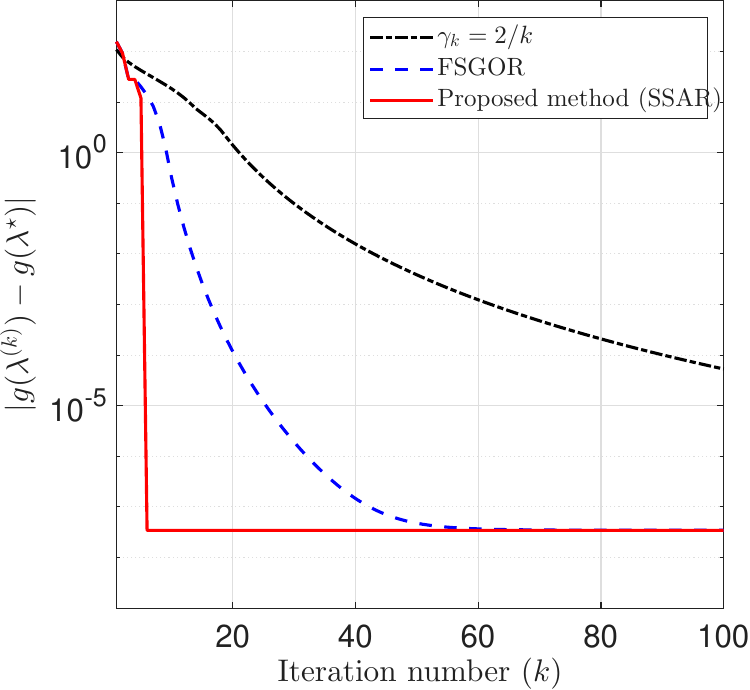}%
\label{subfig:comparison_quadratic_diminishing_FSGOR_AdFSGOR_m_3_n_1.eps}}
\hfil
\caption{Motivation for the SSAR method: path of
$\boldsymbol{\lambda}^{(k)}$ and convergence of the dual function values. \protect\subref{subfig:path_of_lambda_diminishing_stepsize} Dual subgradient method with stepsize $\gamma_k = 2/k$. \protect\subref{subfig:path_of_lambda_FSGOR_initial_stepsize_2.eps}
FSGOR method~\cite{abeynanda2026characteristics}. \protect\subref{subfig:path_of_lambda_AdFSGOR_initial_stepsize_2.eps} SSAR method. \protect\subref{subfig:comparison_quadratic_diminishing_FSGOR_AdFSGOR_m_3_n_1.eps} Convergence of the dual function values.
} 
\label{Fig:motivation-to-AdFSGOR}
\end{figure}

\begin{example}[Motivation to SSAR]\label{ex:motivation-to-AdFSGOR}
Consider the equivalent formulation~\eqref{eq:consensus-problem} of problem~\eqref{eq:main-problem} with $m=3$, $n=1$, $\mathcal{Y}=[-2,2]$, and $f_i(y_i)=y_i^2$ for all $i=1,2,3$. The Fig.~\ref{Fig:motivation-to-AdFSGOR}\subref{subfig:path_of_lambda_diminishing_stepsize}--\subref{subfig:path_of_lambda_AdFSGOR_initial_stepsize_2.eps} depict the level sets of the dual function associated with~\eqref{eq:consensus-problem} (\cf~\eqref{eq:dual-problem-of-consensus-problem}), as well as the dual variable updates $\boldsymbol{\lambda}^{(k)}$ obtained using the dual subgradient method with the stepsize rule $\gamma_k = 2/k$, the FSGOR method~\cite{abeynanda2026characteristics}, and the proposed SSAR method, respectively. The Fig.~\ref{Fig:motivation-to-AdFSGOR}\subref{subfig:path_of_lambda_diminishing_stepsize} shows that the dual subgradient method takes $1777$ steps to converge to the dual optimal solution $\boldsymbol{\lambda}^\star$~\footnote{We consider convergence of $\boldsymbol{\lambda}^{(k)}$ to $\boldsymbol{\lambda}^\star$ with accuracy $10^{-5}$.} using $\gamma_k = 2/k$. Figures~\ref{Fig:motivation-to-AdFSGOR}\subref{subfig:path_of_lambda_FSGOR_initial_stepsize_2.eps} and~\ref{Fig:motivation-to-AdFSGOR}\subref{subfig:path_of_lambda_AdFSGOR_initial_stepsize_2.eps} show that $\mathcal{R}_1$ and $\mathcal{R}_2$ are RFSGOR rays. For both FSGOR and SSAR, the initial constant stepsize is chosen as $\gamma_1 = 2$ as long as $\boldsymbol{\lambda}^{(k)} \in \mathcal{R}_1$, and the stepsize rule $\gamma_k = 1/k$ is used as the switching stepsize rule. From Figure~\ref{Fig:motivation-to-AdFSGOR}\subref{subfig:path_of_lambda_FSGOR_initial_stepsize_2.eps}, we observe that $\boldsymbol{\lambda}^{(3)}$ exits the ray $\mathcal{R}_1$ but lies on $\mathcal{R}_2$. However, the FSGOR method does not detect that $\mathcal{R}_2$ is an RFSGOR ray and therefore switches to the stepsize rule $\gamma_k = 1/k$. It then takes $117$ steps to converge to $\boldsymbol{\lambda}^\star$. In contrast, the proposed approach (\cf~Figure~\ref{Fig:motivation-to-AdFSGOR}\subref{subfig:path_of_lambda_AdFSGOR_initial_stepsize_2.eps}) detects that $\mathcal{R}_2$ is also an RFSGOR ray and again employs a constant stepsize $\gamma_2$ as long as $\boldsymbol{\lambda}^{(k)} \in \mathcal{R}_2$. Here, we choose $\gamma_2 = 0.5$, and Figure~\ref{Fig:motivation-to-AdFSGOR}\subref{subfig:path_of_lambda_AdFSGOR_initial_stepsize_2.eps} shows that the proposed method takes only $5$ steps to converge to $\boldsymbol{\lambda}^\star$. This demonstrates that the proposed method can significantly improve the convergence of the dual subgradient method. Figure~\ref{Fig:motivation-to-AdFSGOR}\subref{subfig:comparison_quadratic_diminishing_FSGOR_AdFSGOR_m_3_n_1.eps} depicts the convergence of the corresponding dual function values for the three methods.
\end{example}

\subsection{The SSAR Algorithm for Solving~\eqref{eq:FSGOR-main-problem}}\label{subsec:Development-of-the-AdFSGOR-Method}
Suppose the SSAR method identifies a finite sequence of RFSGOR rays $\{\mathcal{R}_j\}_{j=1}^N$, where $N\leq N_{\max}$ and $N_{\max}$ denotes the prescribed maximum number of RFSGOR rays to be exploited. The method employs a sequence of decreasing constant stepsizes $\{\gamma_j\}_{j=1}^N$, where $\gamma_j$ is used while $\boldsymbol{\lambda}^{(k)}$ lies on $\mathcal{R}_j$. In particular, we propose the following stepsize rule:
\begin{equation}\label{eq:stepsize-SSAR}
\gamma_k=
\begin{cases}
\displaystyle
\gamma_j=q^{j-1}\gamma
& ; \ \boldsymbol{\lambda}^{(k)}\in\mathcal{R}_j,\
j\in\{1,\ldots,N\}\\
\displaystyle
\hat{\gamma}_k
& ; \ \text{otherwise},
\end{cases}
\end{equation}
where $q\in(0,1)$, $\gamma>0$, and $\hat{\gamma}_k$ denotes an existing stepsize rule that guarantees convergence of the standard dual subgradient method~\eqref{eq:FSGOR-subgradient-method}. The prescribed maximum $N_{\max}$ prevents $\gamma_j$ from becoming excessively small as the number of identified RFSGOR rays increases and limits the SSAR phase to a finite number of rays. While the dual variable iterates remain on each identified RFSGOR ray, the corresponding constant stepsize yields constant decrements in $-g$, enabling the iterates to progress more quickly toward a dual optimal solution $\boldsymbol{\lambda}^\star$.

To employ the stepsize rule~\eqref{eq:stepsize-SSAR}, a mechanism for identifying the RFSGOR rays $\mathcal{R}_j$ is required. The initial ray $\mathcal{R}_1$ is determined as in the FSGOR method~\cite{abeynanda2026characteristics} by solving~\eqref{eq:RFSGOR-feasibility-problem-relaxed}, \cf~Lemma~\ref{lemma:feasibility-of-relaxed-problem-implies-feasibility-of-original-problem}. After $\boldsymbol{\lambda}^{(k)}$ exits an RFSGOR ray $\mathcal{R}_j$, we next seek to determine whether a subsequent RFSGOR ray $\mathcal{R}_{j+1}$ can be identified, where $j\in\{1,\ldots,N-1\}$. To establish a condition for identifying such a ray, we first provide the following result.
\begin{lemma}\label{lemma:proof-affine-over-the-line-segment}
    Let $f:\R^n\to\overline{\R}$ be a convex function. Let $\mathbf{x}_1,\mathbf{x}_2\in\texttt{dom}~f$ with $\mathbf{x}_1\neq\mathbf{x}_2$, and suppose that $\exists~\mathbf{s}\in\R^n$ s.t., $\mathbf{s}\in\partial f(\mathbf{x}_1)$ and $\mathbf{s}\in\partial f(\mathbf{x}_2)$. Then, $f$ is affine over the line segment joining $\mathbf{x}_1$ and $\mathbf{x}_2$, i.e., over the region $\mathcal{C}=\texttt{con}~\{\mathbf{x}_1,\mathbf{x}_2\}$.
\end{lemma}
\begin{proof}
    The proof is given in Appendix~\ref{Appendix:proof-affine-over-the-line-segment}.
\end{proof}
Then, the following result provides a condition for identifying a subsequent RFSGOR ray containing $\boldsymbol{\lambda}^{(k)}$.
\begin{prop}\label{prop:existance-of-RFSGOR-rays-during-iterations}
Suppose Assumption~\ref{Assumption:about-h0-and-X} and
Assumption~\ref{Assumption:about-subdifferential-of-h-at-boundary-ponts}
hold. Let $\{\boldsymbol{\lambda}^{(k)}\}$ be the sequence of dual variable iterates generated by the dual subgradient method~\eqref{eq:FSGOR-subgradient-method},
and let $\mathbf{x}^{(k)}$ be a minimizer of the Lagrangian associated with
problem~\eqref{eq:FSGOR-main-problem_with-indicator-function} corresponding to
the dual variable iterate $\boldsymbol{\lambda}^{(k)}$, \cf~\eqref{eq:minimizer-of-the-Lagrangian}. Suppose that $\mathbf{x}^{(k)}\in\texttt{bnd}~\mathcal{X}$ and problem~\eqref{eq:RFSGOR-feasibility-problem-relaxed} is feasible
with $\mathbf{\bar x}=\mathbf{x}^{(k)}$. Then, $\boldsymbol{\lambda}^{(k)}$ lies on an RFSGOR ray.
\end{prop}
\begin{proof}
Since problem~\eqref{eq:RFSGOR-feasibility-problem-relaxed} is feasible
with $\mathbf{\bar x}=\mathbf{x}^{(k)}$, it follows from
Lemma~\ref{lemma:feasibility-of-relaxed-problem-implies-feasibility-of-original-problem}
and Remark~\ref{Remark:feasibility-of-relaxed-problem} that, for sufficiently
large $c>0$, $\boldsymbol{\hat\mu}=\boldsymbol{\tilde\mu}=(1/\beta)
(\mathbf{A}\mathbf{x}^{(k)}-\mathbf{b})$ and $\boldsymbol{\hat\lambda}=c\boldsymbol{\tilde\mu}=(c/\beta)
(\mathbf{A}\mathbf{x}^{(k)}-\mathbf{b})$ are feasible for problem~\eqref{eq:RFSGOR-feasibility-problem-original}. Then,
\begin{equation}
\label{eq:prop:existance-of-RFSGOR-rays-during-iterations-1}
\hat{\mathcal{R}}(\mathbf{x}^{(k)})
=
\Big\{
\boldsymbol{\hat\lambda}
+
\alpha\boldsymbol{\hat\mu}
\mid
\alpha\geq0
\Big\}
=
\left\{
\frac{c+\alpha}{\beta}
(\mathbf{A}\mathbf{x}^{(k)}{-}\mathbf{b})
\mid
\alpha\geq0
\right\}
\end{equation}
is an RFSGOR ray associated with problem~\eqref{eq:FSGOR-main-problem},
\cf~\eqref{eq:RFSGOR-ray-original}. This means that there exists
$\boldsymbol{a}\in\R^p$ such that $\boldsymbol{a}$ is a subgradient of $-g$ along
the ray $\hat{\mathcal{R}}(\mathbf{x}^{(k)})$ and aligns with the ray.
In particular, we have $\boldsymbol{a}
=\beta\boldsymbol{\hat\mu}=
\mathbf{A}\mathbf{x}^{(k)}-\mathbf{b}$. Moreover, from \eqref{eq:minimizer-of-the-Lagrangian} and the dual function $g$ associated with~\eqref{eq:FSGOR-main-problem_with-indicator-function}, we have $(\mathbf{A}\mathbf{x}^{(k)}-\mathbf{b})
\in
\partial l(\boldsymbol{\lambda}^{(k)})$, where $l=-g$. Thus,
\begin{equation}
\label{eq:prop:existance-of-RFSGOR-rays-during-iterations-2}
\mathbf{A}\mathbf{x}^{(k)}{-}\mathbf{b}
{\in}
\partial l(\boldsymbol{\lambda}^{(k)}) \ \text{and} \ 
\forall~\boldsymbol{\lambda}\in
\hat{\mathcal{R}}(\mathbf{x}^{(k)}),
\mathbf{A}\mathbf{x}^{(k)}{-}\mathbf{b}
{\in}
\partial l(\boldsymbol{\lambda}).
\end{equation}
Equation~\eqref{eq:prop:existance-of-RFSGOR-rays-during-iterations-2}
means that the negative dual function $l$ has the same subgradient at
$\boldsymbol{\lambda}^{(k)}$ and at every point along the ray
$\hat{\mathcal{R}}(\mathbf{x}^{(k)})$. Consider the following two cases: 1) $\boldsymbol{\lambda}^{(k)}
\in\hat{\mathcal{R}}(\mathbf{x}^{(k)})$ and 2) $\boldsymbol{\lambda}^{(k)}
\not\in\hat{\mathcal{R}}(\mathbf{x}^{(k)})$. If
$\boldsymbol{\lambda}^{(k)}\in\hat{\mathcal{R}}(\mathbf{x}^{(k)})$, then it is
straightforward to see that $\boldsymbol{\lambda}^{(k)}$ lies on an
RFSGOR ray, \cf~\eqref{eq:prop:existance-of-RFSGOR-rays-during-iterations-1}.
In the latter case, for every
$\boldsymbol{\lambda}\in\hat{\mathcal{R}}(\mathbf{x}^{(k)})$, the function
$l$ is affine over
$\texttt{con}~\{\boldsymbol{\lambda}^{(k)},\boldsymbol{\lambda}\}$,
\cf~\eqref{eq:prop:existance-of-RFSGOR-rays-during-iterations-2} and
Lemma~\ref{lemma:proof-affine-over-the-line-segment}. Thus, $l$ is affine
over the region
\begin{equation}
\label{eq:prop:existance-of-RFSGOR-rays-during-iterations-3}
\mathcal{C}
=
\texttt{con}
\left(
\{\boldsymbol{\lambda}^{(k)}\}
\cup
\hat{\mathcal{R}}(\mathbf{x}^{(k)})
\right).
\end{equation}
Now consider any point
$\boldsymbol{\lambda}\in\hat{\mathcal{R}}(\mathbf{x}^{(k)})$. Then $\boldsymbol{\lambda}=\boldsymbol{\hat\lambda}+\alpha\boldsymbol{\hat\mu},\alpha\geq0$, \cf~\eqref{eq:prop:existance-of-RFSGOR-rays-during-iterations-1}.
Any point on the line segment joining $\boldsymbol{\lambda}^{(k)}$ and
$\boldsymbol{\lambda}$ is given by
\begin{equation}\label{eq:point-on-the-line-segment}
\mathbf{p}(\alpha,r)
=
(1-r)\boldsymbol{\lambda}^{(k)}
+
r\boldsymbol{\lambda}
=
\boldsymbol{\lambda}^{(k)}+r(\boldsymbol{\hat\lambda}+\alpha\boldsymbol{\hat\mu}-\boldsymbol{\lambda}^{(k)}),
\end{equation}
where $r\in[0,1]$. Clearly,
\begin{equation}
\label{eq:point-between-a-line-segment-3}
\forall~\alpha\geq0,\ \forall~r\in[0,1],\quad
\mathbf{p}(\alpha,r)\in\mathcal{C},
\end{equation}
where $\mathcal{C}$ is given in
\eqref{eq:prop:existance-of-RFSGOR-rays-during-iterations-3}. Now, for any fixed $t\geq0$ and any $\alpha>t$, choose
$r=t/\alpha$. Then, from~\eqref{eq:point-on-the-line-segment},
\begin{equation}
\label{eq:point-between-a-line-segment-4}
\mathbf{p}(\alpha,r)=\mathbf{p}(\alpha,t/\alpha)=\boldsymbol{\lambda}^{(k)}+
\frac{t}{\alpha}(
\boldsymbol{\hat\lambda}-\boldsymbol{\lambda}^{(k)})+t\boldsymbol{\hat\mu}.
\end{equation}
Therefore,
\begin{equation}
\label{eq:point-between-a-line-segment-5}
\lim_{\alpha\to\infty}
\mathbf{p}(\alpha,t/\alpha)
=
\boldsymbol{\lambda}^{(k)}
+
t\boldsymbol{\hat\mu}.
\end{equation}
Since $\mathcal{C}$ is closed and
$\mathbf{p}(\alpha,t/\alpha)\in\mathcal{C}$ for all $\alpha>t$, it follows that
$\boldsymbol{\lambda}^{(k)}+t\boldsymbol{\hat\mu}\in\mathcal{C}$ for all $t\geq0$.
Thus, $\mathcal{R}=\left\{\boldsymbol{\lambda}^{(k)}+t\boldsymbol{\hat\mu}\mid t\geq 0\right\}\subseteq\mathcal{C}$. Hence, $l$ is affine over $\mathcal{R}$, with $\mathbf{A}\mathbf{x}^{(k)}-\mathbf{b}
\in\partial l(\boldsymbol{\lambda}),\,\forall~\boldsymbol{\lambda}\in\mathcal{R}$,
\cf~\eqref{eq:prop:existance-of-RFSGOR-rays-during-iterations-2} and
\eqref{eq:prop:existance-of-RFSGOR-rays-during-iterations-3}.
Moreover, we have $\mathbf{A}\mathbf{x}^{(k)}-\mathbf{b}=\beta\boldsymbol{\hat\mu}$, \cf~\eqref{eq:prop:existance-of-RFSGOR-rays-during-iterations-1}. Therefore, $\mathcal{R}$ is an RFSGOR ray. Hence, the proof follows because $\boldsymbol{\lambda}^{(k)}\in\mathcal{R}$. 
\end{proof}

The corresponding SSAR algorithm is provided in Algorithm~\ref{Alg:AdFSGOR-algorithm}, and its steps are described below.
\subsubsection{Description of Algorithm~\ref{Alg:AdFSGOR-algorithm}}
First, problem~\eqref{eq:RFSGOR-feasibility-problem-relaxed} is required to be feasible for some $\mathbf{\bar x}\in\texttt{bnd}~\mathcal{X}$, and the corresponding solution $\boldsymbol{\tilde\mu}$ is obtained. The initial RFSGOR ray $\mathcal{R}_1$ is then obtained according to Lemma~\ref{lemma:feasibility-of-relaxed-problem-implies-feasibility-of-original-problem}, and the dual subgradient method~\eqref{eq:FSGOR-subgradient-method} is initialized from $\mathcal{R}_1$, \cf~step~$1$. The primal variable iterate $\mathbf{x}^{(k)}$ is computed in step~$2$, \cf~\eqref{eq:minimizer-of-the-Lagrangian}. As long as $\boldsymbol{\lambda}^{(k)}\in\mathcal{R}_j$ (\cf~\cite[Remark~4]{abeynanda2026characteristics}), the dual subgradient method is performed with $\gamma_k=q^{j-1}\gamma$ and the corresponding subgradient $\mathbf{d}^{(k)}=\beta\boldsymbol{\tilde\mu}$ of $-g$, \cf~steps~$3$--$5$. Once $\boldsymbol{\lambda}^{(k)}$ exits $\mathcal{R}_j$, step~$7$ checks the sufficient conditions in Proposition~\ref{prop:existance-of-RFSGOR-rays-during-iterations} for identifying a subsequent RFSGOR ray $\mathcal{R}_{j+1}$ containing $\boldsymbol{\lambda}^{(k)}$. If these conditions are satisfied, the direction $\boldsymbol{\tilde\mu}$ of $\mathcal{R}_{j+1}$ is computed, and the dual subgradient method is performed with $\gamma_k=q^j\gamma$, \cf~steps~$8$--$10$. This process is repeated until no subsequent RFSGOR ray is identified or $N_{\max}$ RFSGOR rays have been identified. The SSAR exploitation phase then terminates, and the dual subgradient method proceeds with the existing stepsize rule $\hat{\gamma}_k$ until a stopping criterion is satisfied, \cf~steps~$15$--$18$.

\begin{algorithm}[t]
	\caption{SSAR Algorithm for Solving~\eqref{eq:FSGOR-main-problem}} 
	\begin{algorithmic}[1]
            \Require A solution $\tilde{\boldsymbol{\mu}}$ obtained from problem~\eqref{eq:RFSGOR-feasibility-problem-relaxed} corresponding to some $\mathbf{\bar x}\in\texttt{bnd}~\mathcal{X}$ (\cf~Remark~\ref{remark:condition-for-choosing-x-bar}), $c >0$, $\beta>0$, $q\in(0,1)$, $\gamma>0$, and $N_{\max}\in\mathbb{Z}_+$.  
	    \State $\boldsymbol{\lambda}^{(0)}=c\tilde{\boldsymbol{\mu}}$, $k=0$, $j=1$.
        \State Compute: ${\mathbf{x}^{(k)}{\in}
\arg\min_{\mathbf{x}\in\mathcal{X}}(h_0(\mathbf{x}){-}\boldsymbol{\lambda}^{{(k)}\text{T}}\mathbf{A}\mathbf{x}{+}\boldsymbol{\lambda}^{{(k)}\text{T}}\mathbf{b})}$.
        \If{$\mathbf{x}^{(k)}=\mathbf{\bar x}$} \Comment{$\boldsymbol{\lambda}^{(k)}$ lies on RFSGOR ray $\mathcal{R}_j$}

                \State Perform \eqref{eq:FSGOR-subgradient-method} with $\gamma_k=q^{j-1}\gamma$ and $\mathbf{d}^{(k)}=\beta\boldsymbol{\tilde\mu}$.
                \State Set $k:=k+1$ and go to step~2.
            \Else \Comment{$\boldsymbol{\lambda}^{(k)}$ exits RFSGOR ray $\mathcal{R}_j$}
            \If{$j<N_{\max}$, $\mathbf{x}^{(k)}\in\texttt{bnd}~\mathcal{X}$, and problem~\eqref{eq:RFSGOR-feasibility-problem-relaxed} is feasible with $\mathbf{\bar x}=\mathbf{x}^{(k)}$} \Comment{$\boldsymbol{\lambda}^{(k)}$ lies on $\mathcal{R}_{j+1}$}
            \State Set $\boldsymbol{\tilde\mu}:=(\mathbf{A}\mathbf{x}^{(k)}-\mathbf{b})/\beta$ and $\mathbf{\bar x}:=\mathbf{x}^{(k)}$. 
            \State Perform \eqref{eq:FSGOR-subgradient-method} with $\gamma_k=q^j\gamma$ and $\mathbf{d}^{(k)}=\beta\boldsymbol{\tilde\mu}$.
            \State Set $k:=k+1$, $j:=j+1$, and go to step~$2$. 
            \Else \Comment{The SSAR exploitation phase terminates}
            \State Go to step~$15$.
            \EndIf 
            \EndIf
            \Repeat 
            \State Perform \eqref{eq:FSGOR-subgradient-method} with an existing stepsize rule $\hat{\gamma}_k$, \cf~\eqref{eq:stepsize-SSAR}.
            \State $k:=k+1$.
        \Until{a stopping criterion is satisfied}
	\end{algorithmic} 
	\label{Alg:AdFSGOR-algorithm}
\end{algorithm} 

The convergence of Algorithm~\ref{Alg:AdFSGOR-algorithm} is established below.
\begin{prop}
Suppose Assumption~\ref{Assumption:about-h0-and-X} and
Assumption~\ref{Assumption:about-subdifferential-of-h-at-boundary-ponts}
hold. Then, Algorithm~\ref{Alg:AdFSGOR-algorithm} converges to an optimal
solution $\boldsymbol{\lambda}^\star$ of the dual
problem~\eqref{eq:FSGOR-dual-problem}.
\end{prop}
\begin{proof}
Consider an RFSGOR ray $\mathcal{R}_j$ in
Algorithm~\ref{Alg:AdFSGOR-algorithm}, where $j\in\{1,\ldots,N\}$, and let
$k_j$ denote the iteration at which the algorithm starts exploiting
$\mathcal{R}_j$. Let $\mathbf{d}_j$ denote the subgradient of the negative
dual function $-g$ that remains constant and aligns with
$\mathcal{R}_j$. While $\boldsymbol{\lambda}^{(k)}$ lies on
$\mathcal{R}_j$, Algorithm~\ref{Alg:AdFSGOR-algorithm} uses the constant
stepsize $\gamma_j=q^{j-1}\gamma$ (\cf~\eqref{eq:stepsize-SSAR}) and
performs the update $\boldsymbol{\lambda}^{(k+1)}
=\boldsymbol{\lambda}^{(k)}-\gamma_j\mathbf{d}_j$. There are two possible
cases: 1) $\mathbf{d}_j=\mathbf{0}_p$ and 2)
$\mathbf{d}_j\neq\mathbf{0}_p$. In the first case,
$\boldsymbol{\lambda}^{(k)}$ is an optimal solution of
\eqref{eq:FSGOR-dual-problem}, since
$\mathbf{0}_p\in\partial(-g)(\boldsymbol{\lambda}^{(k)})$. In the latter
case, as long as consecutive iterates remain on $\mathcal{R}_j$, we have
\begin{equation}\label{eq:proof-SSAR-1}
-g(\boldsymbol{\lambda}^{(k+1)})-g(\boldsymbol{\lambda}^{(k)})
+
\mathbf{d}_j\tran
\bigl(
\boldsymbol{\lambda}^{(k+1)}-\boldsymbol{\lambda}^{(k)}
\bigr),
\end{equation}
because $-g$ is affine along $\mathcal{R}_j$ with the corresponding
constant subgradient $\mathbf{d}_j$. Then, using
$\boldsymbol{\lambda}^{(k+1)}
=\boldsymbol{\lambda}^{(k)}-\gamma_j\mathbf{d}_j$,
\eqref{eq:proof-SSAR-1} yields
\begin{equation}\label{eq:proof-SSAR-2}
-g(\boldsymbol{\lambda}^{(k+1)})-g(\boldsymbol{\lambda}^{(k)})
-\gamma_j\|\mathbf{d}_j\|_2^2.
\end{equation}
From~\eqref{eq:proof-SSAR-2}, we observe that $-g$ decreases by the
constant amount $\gamma_j\|\mathbf{d}_j\|_2^2>0$ at each iteration while
consecutive iterates remain on $\mathcal{R}_j$. Suppose, for contradiction, that the dual variable iterates remain on
$\mathcal{R}_j$ indefinitely. Then, by recursively applying
\eqref{eq:proof-SSAR-2}, we obtain
\begin{equation}\label{eq:proof-SSAR-3}
-g(\boldsymbol{\lambda}^{(k_j+k)})
-g(\boldsymbol{\lambda}^{(k_j)})
-k\gamma_j\|\mathbf{d}_j\|_2^2.
\end{equation}
It follows from~\eqref{eq:proof-SSAR-3} that
$-g(\boldsymbol{\lambda}^{(k_j+k)})\rightarrow-\infty$ as
$k\rightarrow\infty$. On the other hand, by weak duality,
$g(\boldsymbol{\lambda})\leq p^\star$ for every
$\boldsymbol{\lambda}\in\R^p$, where $p^\star$ denotes the finite optimal
value of the primal problem~\eqref{eq:FSGOR-main-problem}. Consequently,
$-g(\boldsymbol{\lambda})\geq-p^\star>-\infty$ for every
$\boldsymbol{\lambda}\in\R^p$, which contradicts
$-g(\boldsymbol{\lambda}^{(k_j+k)})\rightarrow-\infty$. Therefore, the
dual variable iterates cannot remain on $\mathcal{R}_j$ indefinitely.
Hence, for every RFSGOR ray identified by the SSAR algorithm, either an
optimal dual solution $\boldsymbol{\lambda}^\star$ is reached or the dual
variable iterates exit the ray after a finite number of iterations.

After the dual variable iterates exit $\mathcal{R}_j$, if $j<N$, the
SSAR algorithm checks whether a subsequent RFSGOR ray can be identified
according to
Proposition~\ref{prop:existance-of-RFSGOR-rays-during-iterations}. If no subsequent RFSGOR ray is identified, the SSAR phase terminates, and the algorithm proceeds with the standard dual subgradient method using an existing stepsize rule $\hat{\gamma}_k$, \cf~\eqref{eq:stepsize-SSAR}. If a subsequent RFSGOR ray $\mathcal{R}_{j+1}$ is identified, the SSAR
algorithm exploits $\mathcal{R}_{j+1}$ using the constant stepsize
$\gamma_{j+1}=q^j\gamma$, and the preceding argument for
$\mathcal{R}_j$ applies to $\mathcal{R}_{j+1}$. Since the SSAR algorithm
exploits at most $N$ RFSGOR rays and the dual variable iterates remain on
each identified RFSGOR ray for only a finite number of iterations unless
an optimal dual solution is reached, the SSAR phase terminates after a
finite number of iterations. Thus, if an optimal dual solution has not been reached during the SSAR phase, then, after a finite number of iterations, the algorithm proceeds with the standard dual subgradient method using an existing stepsize rule \(\hat{\gamma}_k\), which guarantees convergence of the dual subgradient method, \(\cf\)~\eqref{eq:stepsize-SSAR}. Therefore, $\boldsymbol{\lambda}^{(k)}$
converges to an optimal solution $\boldsymbol{\lambda}^\star$ of
\eqref{eq:FSGOR-dual-problem}. This completes the proof.
\end{proof}

Then, a primal optimal solution of problem~\eqref{eq:FSGOR-main-problem} can be recovered from a dual optimal solution obtained by Algorithm~\ref{Alg:AdFSGOR-algorithm}, as stated in the following corollary.

\begin{corollary}
Suppose Assumptions~\ref{Assumption:about-h0-and-X} and
\ref{Assumption:about-subdifferential-of-h-at-boundary-ponts} hold, and
strong duality holds. Let $\boldsymbol{\lambda}^\star$ be a dual optimal
solution obtained by Algorithm~\ref{Alg:AdFSGOR-algorithm}. If
$\mathbf{x}^\star\in\arg\min_{\mathbf{x}\in\mathcal{X}}
L(\mathbf{x},\boldsymbol{\lambda}^\star)$ satisfies
$\mathbf{A}\mathbf{x}^\star=\mathbf{b}$, then $\mathbf{x}^\star$ is an
optimal solution of problem~\eqref{eq:FSGOR-main-problem}.
\end{corollary}

\begin{proof}
Since $\mathbf{x}^\star$ is a Lagrangian minimizer associated with
$\boldsymbol{\lambda}^\star$, we have
$L(\mathbf{x}^\star,\boldsymbol{\lambda}^\star)
=g(\boldsymbol{\lambda}^\star)$. Moreover, since
$\mathbf{A}\mathbf{x}^\star=\mathbf{b}$, we have
$L(\mathbf{x}^\star,\boldsymbol{\lambda}^\star)
=h_0(\mathbf{x}^\star)$. By strong duality,
$g(\boldsymbol{\lambda}^\star)=p^\star$. Therefore,
$h_0(\mathbf{x}^\star)=p^\star$, and hence $\mathbf{x}^\star$ is an
optimal solution of problem~\eqref{eq:FSGOR-main-problem}.
\end{proof}

\section{SSAR-Assisted Dual Decomposition Method and Computation of Initial RFSGOR Rays}
\label{sec:AdFSGOR-Assisted-Dual-Decomposition-Method-and-computation-of-RFSGOR-rays}
In this section, we first discuss how Algorithm~\ref{Alg:AdFSGOR-algorithm} can be adopted to solve problem~\eqref{eq:main-problem} using dual decomposition. Then, we discuss how to determine initial RFSGOR rays associated with the consensus problem~\eqref{eq:consensus-problem}, which are used to initialize the SSAR-assisted dual decomposition algorithm.

\subsection{SSAR-Assisted Dual Decomposition Method}\label{subsec:AdFSGOR-Assisted-Dual-Decomposition-Method}

We consider the equivalent formulation
\eqref{eq:consensus-problem} of problem~\eqref{eq:main-problem}.
First, we derive the relaxed feasibility problem~\eqref{eq:RFSGOR-feasibility-problem-relaxed} associated with~\eqref{eq:consensus-problem}. Note that problem~\eqref{eq:consensus-problem} has the form
\eqref{eq:FSGOR-main-problem}, with
$\mathbf{x}=\mathbf{y}
=[\mathbf{y}_{1}\tran \ \ldots \ \mathbf{y}_{m}\tran]\tran$,
$h_0=f$, $\mathcal{X}=\mathcal{Y}^m$, $\mathbf{A}=\mathbf{C}$
(\cf~\eqref{eq:consenses-problem-matrix-C}), and
$\mathbf{b}=\mathbf{0}_{n(m-1)}$.
Thus, the relaxed feasibility problem is
\begin{equation}
\label{eq:RFSGOR-feasibility-problem-relaxed-associated-with-consensus-problem}
\begin{array}{ll}
\underset{(\boldsymbol{\mu},\boldsymbol{\eta})}{\text{minimize}} & 0\\
\mbox{subject to} &
\mathbf{C}\mathbf{\bar y}
=
\beta(\mathbf{C}\mathbf{C}\tran)^{-1}\mathbf{C}{\boldsymbol\eta}\\
&
{\boldsymbol\eta} \in \bar{N}_{\mathcal{Y}^m}(\mathbf{\bar y})\\
&
\mathbf{C}\tran \boldsymbol{\mu} = {\boldsymbol\eta},
\end{array}
\end{equation}
where $\mathbf{\bar y}\in\texttt{bnd}~\mathcal{Y}^m$ and
$\bar{N}_{\mathcal{Y}^m}(\mathbf{\bar y})
=\texttt{int}~N_{\mathcal{Y}^m}(\mathbf{\bar y})$.
According to Lemma~\ref{lemma:feasibility-of-relaxed-problem-implies-feasibility-of-original-problem},
Remark~\ref{Remark:feasibility-of-relaxed-problem}, and
Remark~\ref{remark:condition-for-choosing-x-bar}, an RFSGOR ray
$\hat{\mathcal{R}}(\mathbf{\bar y})$
(\cf~\eqref{eq:RFSGOR-ray-original}) associated with
problem~\eqref{eq:consensus-problem} can be computed if
$\boldsymbol{\tilde\eta}
=\mathbf{C}\tran\mathbf{C}\mathbf{\bar y}/\beta
\in\texttt{int}~N_{\mathcal{Y}^m}(\mathbf{\bar y})$
for some $\mathbf{\bar y}\in\texttt{bnd}~\mathcal{Y}^m$ satisfying
$\texttt{aff}(N_{\mathcal{Y}^m}(\mathbf{\bar y}))=\R^{mn}$.
However, characterizing the normal cone of the Cartesian product
$\mathcal{Y}^m$ may not be straightforward, even when $\mathcal{Y}$
has a simple form. Therefore, we provide a few results below that
facilitate working with the normal cones of $\mathcal{Y}$ instead of
the normal cone of $\mathcal{Y}^m$.

\begin{lemma}\label{lemma:normal-cone-cartesian-product-results}
Let $\mathbf{\bar y}_i\in\texttt{bnd}~\mathcal{Y}$ for
$i\in\{1,\ldots,m\}$, let
$\mathbf{\bar y}=[\mathbf{\bar y}_1\tran\ldots\mathbf{\bar y}_m\tran]\tran$,
and let
$\boldsymbol{\eta}=[\boldsymbol{\eta}_1\tran\ldots\boldsymbol{\eta}_m\tran]\tran$.
Then the following results hold.
\begin{enumerate}[a)]
    \item
    $N_{\mathcal{Y}^m}(\mathbf{\bar y})
    =
    N_{\mathcal{Y}}(\mathbf{\bar y}_1)\times\ldots\times
    N_{\mathcal{Y}}(\mathbf{\bar y}_m)$
    \item
    $\texttt{int}~N_{\mathcal{Y}^m}(\mathbf{\bar y})
    =
    \texttt{int}~N_{\mathcal{Y}}(\mathbf{\bar y}_1)\times\ldots\times
    \texttt{int}~N_{\mathcal{Y}}(\mathbf{\bar y}_m)$
    
    \item
    $\boldsymbol{\eta}\in
    \texttt{int}~N_{\mathcal{Y}^m}(\mathbf{\bar y})
    \iff \boldsymbol{\eta}_i\in
    \texttt{int}~N_{\mathcal{Y}}(\mathbf{\bar y}_i)$, $\forall~i\in\mathcal{S}$
        
    \item
    $\forall~i{\in}\mathcal{S},\,\texttt{aff}(N_{\mathcal{Y}}(\mathbf{\bar y}_i))=\R^n 
    {\iff}
    \texttt{aff}(N_{\mathcal{Y}^m}(\bar{\mathbf{y}}))=\R^{mn}$
\end{enumerate}
\end{lemma}
\begin{proof}\label{proof:normal-cone-cartesian-product-results}
For $\mathbf{y}=[\mathbf{y}_1\tran\ldots\mathbf{y}_m\tran]\tran$, we have
$\delta_{\mathcal{Y}^m}(\mathbf{y})
=
\sum_{i=1}^m\delta_{\mathcal{Y}}(\mathbf{y}_i)$.
Therefore, $\partial\delta_{\mathcal{Y}^m}(\mathbf{\bar y})
=
\prod_{i=1}^m
\partial\delta_{\mathcal{Y}}(\mathbf{\bar y}_i)$. Then, the first result follows from
$\partial\delta_{\mathcal{X}}(\mathbf{\bar x})
=
N_{\mathcal{X}}(\mathbf{\bar x})$
for any convex set $\mathcal{X}\subseteq\R^n$ and
$\mathbf{\bar x}\in\mathcal{X}$.
The second result follows from the first result and $\texttt{int}~(\mathcal{A}_1\times\ldots\times\mathcal{A}_m)
=
\prod_{i=1}^m\texttt{int}~\mathcal{A}_i$
for any sets
$\mathcal{A}_i\subseteq\R^n$.
The third result follows directly from the second result and the
definition of the Cartesian product of sets.
The fourth result follows from the first result together with
$\texttt{aff}(\mathcal{B}_1\times\ldots\times\mathcal{B}_m)
=
\prod_{i=1}^m\texttt{aff}(\mathcal{B}_i)$
for any nonempty sets $\mathcal{B}_i\subseteq\R^n$.
\end{proof}

From this point onward, we consider the constraint set $\mathcal{Y}$ in
problem~\eqref{eq:consensus-problem} to be a bounded polyhedron, since it is common in engineering
applications, including distributed resource allocation, network flow
optimization, economic dispatch, and distributed model predictive control.
Then,
\begin{equation}\label{eq:polyhedron-constraint-set}
\mathcal{Y}=\{\mathbf{z}\in\R^n \mid \mathbf{D}\mathbf{z}\leq \mathbf{c}\},
\end{equation}
where $\mathbf{D}\in\R^{p\times n}$ and $\mathbf{c}\in\R^p$.
For such a polyhedron, a boundary point
$\mathbf{\bar y}_i\in\texttt{bnd}~\mathcal{Y}$ satisfies
$\texttt{aff}(N_{\mathcal{Y}}(\mathbf{\bar y}_i))=\R^n$
if and only if $\mathbf{\bar y}_i$ is a vertex of $\mathcal{Y}$.
Consequently, by the fourth result of
Lemma~\ref{lemma:normal-cone-cartesian-product-results}, we obtain
$\mathbf{\bar y}\in\texttt{bnd}~\mathcal{Y}^m$ satisfying
$\texttt{aff}(N_{\mathcal{Y}^m}(\mathbf{\bar y}))=\R^{mn}$
by selecting each $\mathbf{\bar y}_i$ to be a vertex of $\mathcal{Y}$.
Furthermore, by the third result of
Lemma~\ref{lemma:normal-cone-cartesian-product-results}, we check
whether $\boldsymbol{\tilde\eta}\in\texttt{int}~
N_{\mathcal{Y}^m}(\mathbf{\bar y})$ by checking whether each component $\boldsymbol{\tilde\eta}_i$ belongs to
$\texttt{int}~N_{\mathcal{Y}}(\mathbf{\bar y}_i)$,
for $i\in\{1,\ldots,m\}$.

The implementation of Algorithm~\ref{Alg:AdFSGOR-algorithm} for solving
problem~\eqref{eq:consensus-problem} (equivalently, problem~\eqref{eq:main-problem})
using dual decomposition (\cf~Algorithm~\ref{Alg:Dual-Decomposition-Algorithm})
requires a well-defined communication structure between the CS and the agents
in the distributed system. Specifically, the communication structure is
important for determining whether $\boldsymbol{\lambda}^{(k)}$ lies on an
RFSGOR ray and for employing the stepsize rule~\eqref{eq:stepsize-SSAR}
in the dual decomposition setting. In this respect, we propose
Algorithm~\ref{Alg:AdFSGOR-assisted_dual-decomposition-method}, a variant of
the standard dual decomposition (Algorithm~\ref{Alg:Dual-Decomposition-Algorithm}) that integrates the
SSAR method (Algorithm~\ref{Alg:AdFSGOR-algorithm}) into the decomposition framework. The algorithm steps are described below.

\begin{algorithm}[!t]
	\caption{SSAR-Assisted Dual Decomposition Algorithm} 
	\begin{algorithmic}[1]
	     \Require  A solution $\tilde{\boldsymbol{\mu}}$ obtained from problem~\eqref{eq:RFSGOR-feasibility-problem-relaxed-associated-with-consensus-problem} corresponding to some $\mathbf{\bar y}\in\texttt{ver}~\mathcal{Y}^m$, $c >0$, $\beta>0$, $q\in(0,1)$, $\gamma>0$, and $N_{\max}\in\mathbb{Z}_+$. 
       \State $\forall~i$, agent $i$ and the CS compute $\boldsymbol{\lambda}^{(0)}=c\tilde{\boldsymbol{\mu}}$.
	    \State $k=0$, $j=1$, $\boldsymbol{\lambda}_0^{(r)}=\boldsymbol{\lambda}_m^{(r)}=\mathbf{0}_n, \ r\in\mathbb{Z}^0_+$.
      \State $\forall~i$, agent $i$ computes:  $$\mathbf{y}_i^{(k)}\in\underset{\mathbf{y}_i\in\mathcal{Y}}{\arg\min}~\frac{1}{m}f_i(\mathbf{y}_i)+ {\big(\boldsymbol{\lambda}^{(k)}_i-\boldsymbol{\lambda}^{(k)}_{i-1}\big)}\tran \mathbf{y}_i.$$ 
\State $\forall~i$, agent $i$ transmits one bit $b_i$ to the CS, where $b_i=0$ if $\mathbf{y}_i^{(k)}=\mathbf{\bar y}_i$ and $b_i=1$ otherwise.
\State CS computes $\bar b=\max_{i\in\mathcal{S}} b_i$ and broadcasts $\bar b$ to each agent $i$. 
\If {$\bar b=0$,} \Comment{$\boldsymbol{\lambda}^{(k)}$ lies on RFSGOR ray $\mathcal{R}_j$} 
\State $\forall i$, agent $i$ and the CS perform \eqref{eq:FSGOR-subgradient-method} with $\gamma_k=q^{j-1}\gamma$ and $\mathbf{d}^{(k)}=\beta\boldsymbol{\tilde\mu}$.
\State Set $k:=k+1$ and go to step~3.
\Else  \Comment{$\boldsymbol{\lambda}^{(k)}$ exits RFSGOR ray $\mathcal{R}_j$}
\State $\forall~i$, agent $i$ transmits one bit $b_i$ to the CS, where $b_i=0$ if $\mathbf{y}_i^{(k)}\in\texttt{ver}~\mathcal{Y}$ and $b_i=1$ otherwise.
\State CS computes $\bar b=\max_{i\in\mathcal{S}} b_i$ and broadcasts $\bar b$ to each agent $i$. 
\If {$\bar b=1$,} \Comment{SSAR exploitation phase terminates}
\State $\forall~i$, agent $i$ transmits $\mathbf{ y}_i^{(k)}$ to the CS. 
\State CS performs~\eqref{eq:FSGOR-subgradient-method} with $\gamma_k{=}\hat{\gamma}_k$ (\cf \eqref{eq:stepsize-SSAR}) and $\mathbf{d}^{(k)}=\mathbf{C}\mathbf{y}^{(k)}$.
\State CS broadcasts $\boldsymbol{\lambda}^{(k+1)}$ to each agent $i$ and go to step~$28$. 
\Else \Comment{$\forall~i, \, \mathbf{y}_i^{(k)}\in\texttt{ver}~\mathcal{Y}$}
\State $\forall~i$, agent $i$ transmits the index of $\mathbf{ y}_i^{(k)}$ to the CS. 
\State The CS broadcasts the indices of
$\mathbf{y}_i^{(k)}$, $i\in\mathcal{S}$,
to all agents. \Comment{Then all agents have $\mathbf{y}^{(k)}{=}[\mathbf{y}_{1}^{(k)\text{T}} \ldots\mathbf{y}_{m}^{(k)\text{T}}]\tran$}
\State Both the CS and the agents compute
$\tilde{\boldsymbol{\mu}}
=\mathbf{C}\mathbf{y}^{(k)}/\beta$
and
$\tilde{\boldsymbol{\eta}}
=\mathbf{C}\tran\tilde{\boldsymbol{\mu}}$.
\If {$j<N_{\max}$ and $\forall~i,\,\boldsymbol{\tilde\eta}_i\in\texttt{int}~N_{\mathcal{Y}}(\mathbf{y}_i^{(k)})$,}
\Comment{$\boldsymbol{\lambda}^{(k)}$ lies on RFSGOR ray $\mathcal{R}_{j+1}$ (\cf~Corollary~\ref{cor:existance-of-RFSGOR-rays-during-iterations-consensus-problem})}
\State Both the CS and the agents perform
\eqref{eq:FSGOR-subgradient-method}
with $\gamma_k=q^j\gamma$ and
$\mathbf{d}^{(k)}
=\beta\tilde{\boldsymbol{\mu}}$.
\State $\forall~i$, agent $i$ sets $\mathbf{\bar y}_i=\mathbf{y}_i^{(k)}$, both the CS and the agents set $k:=k+1$, $j:=j+1$, and go to step~$3$.
\Else \Comment{SSAR exploitation phase terminates}  
\State Both the CS and agents perform~\eqref{eq:FSGOR-subgradient-method} with $\gamma_k{=}\hat{\gamma}_k$ (\cf \eqref{eq:stepsize-SSAR}) and $\mathbf{d}^{(k)}=\mathbf{C}\mathbf{y}^{(k)}$, and go to step~$28$.
 \EndIf      %
\EndIf 
\EndIf 
\State Switch to Alg.~\ref{Alg:Dual-Decomposition-Algorithm} with $\boldsymbol{\lambda}{=}\boldsymbol{\lambda}^{(k+1)}$ and $\gamma_k{=}\hat{\gamma}_k$, \cf\eqref{eq:stepsize-SSAR}.
      	\end{algorithmic} 
	\label{Alg:AdFSGOR-assisted_dual-decomposition-method}
\end{algorithm} 
\subsubsection{Description of Algorithm~\ref{Alg:AdFSGOR-assisted_dual-decomposition-method}}
First, the algorithm requires
problem~\eqref{eq:RFSGOR-feasibility-problem-relaxed-associated-with-consensus-problem}
to be feasible for some
$\mathbf{\bar y}\in\texttt{ver}~\mathcal{Y}^m$. An initial RFSGOR ray $\mathcal{R}_1$ associated with
problem~\eqref{eq:consensus-problem} can then be obtained according to
Lemma~\ref{lemma:feasibility-of-relaxed-problem-implies-feasibility-of-original-problem}. Since $\mathbf{C}$, $\mathcal{Y}$, $\mathbf{\bar y}$, and $\beta$ are known
to both the CS and the agents, they can independently solve
\eqref{eq:RFSGOR-feasibility-problem-relaxed-associated-with-consensus-problem}
to obtain $\boldsymbol{\tilde\mu}$ and compute
$\boldsymbol{\lambda}^{(0)}$, \cf~step~$1$. The agents then compute their local solutions $\mathbf{y}_i^{(k)}$,
\cf~step~$3$. In step~$4$, each agent communicates to the CS whether
$\mathbf{y}_i^{(k)}=\mathbf{\bar y}_i$. If $\bar b=0$, then
$\mathbf{y}^{(k)}=\mathbf{\bar y}$, and hence
$\boldsymbol{\lambda}^{(k)}$ lies on $\mathcal{R}_j$, \cf~steps~$5$--$6$. The corresponding
constant stepsize is then employed, \cf~steps~$7$--$8$. Once $\boldsymbol{\lambda}^{(k)}$ exits
$\mathcal{R}_j$ (\cf~step~$9$), the subsequent steps
(\cf~steps~$10$--$23$) check whether a subsequent RFSGOR ray
$\mathcal{R}_{j+1}$ containing $\boldsymbol{\lambda}^{(k)}$ can be
identified. If no subsequent RFSGOR ray is identified or $j=N_{\max}$, the SSAR
exploitation phase terminates (\cf~steps~$12$ and~$23$), and the algorithm
proceeds according to the standard dual decomposition method,
\cf~steps~$13$--$15$ and~$24$. In particular, step~$20$ checks a
sufficient condition for identifying a subsequent RFSGOR ray
$\mathcal{R}_{j+1}$ containing $\boldsymbol{\lambda}^{(k)}$, as established
by the following result.

\begin{corollary}\label{cor:existance-of-RFSGOR-rays-during-iterations-consensus-problem}
Suppose that each $f_i$ in~\eqref{eq:consensus-problem} is convex,
lower semicontinuous, and proper, $\mathcal{Y}$ is a bounded polyhedron, and
Assumption~\ref{Assumption:about-subdifferential-of-h-at-boundary-ponts}
holds with $h_0$ and $\mathcal{X}$ replaced by $f$ and
$\mathcal{Y}^m$, respectively, \cf~\eqref{eq:consensus-problem}.
Let $\{\boldsymbol{\lambda}^{(k)}\}$ be the sequence of dual variable
iterates generated by
Algorithm~\ref{Alg:AdFSGOR-assisted_dual-decomposition-method}, where
$\boldsymbol{\lambda}^{(k)}=
[\boldsymbol{\lambda}_1^{(k)\text{T}}\ \ldots
\boldsymbol{\lambda}_{m-1}^{(k)\text{T}}]\tran$, and let
$\mathbf{y}_i^{(k)}$ be a minimizer of the local subproblem of agent~$i$
corresponding to $\boldsymbol{\lambda}^{(k)}$ for all $i\in\mathcal{S}$,
\cf~\eqref{eq:dual-problem-of-consensus-problem} and step~$3$ of
Algorithm~\ref{Alg:AdFSGOR-assisted_dual-decomposition-method}.
Suppose that
$\mathbf{y}_i^{(k)}\in\texttt{ver}~\mathcal{Y}$ and
$\boldsymbol{\tilde\eta}_i\in
\texttt{int}~N_{\mathcal{Y}}(\mathbf{y}_i^{(k)})$
for all $i\in\mathcal{S}$, where
$\boldsymbol{\tilde\eta}=
[\boldsymbol{\tilde\eta}_1\tran\ldots
\boldsymbol{\tilde\eta}_m\tran]\tran=
\mathbf{C}\tran\mathbf{C}\mathbf{y}^{(k)}/\beta$.
Then $\boldsymbol{\lambda}^{(k)}$ lies on an RFSGOR ray.
\end{corollary}
\begin{proof}
Since $\mathbf{y}_i^{(k)}\in\texttt{ver}~\mathcal{Y}$ for all
$i\in\mathcal{S}$, we have
$\mathbf{y}^{(k)}\in\texttt{ver}~\mathcal{Y}^m\subseteq
\texttt{bnd}~\mathcal{Y}^m$, where
$\mathbf{y}^{(k)}=
[\mathbf{y}_1^{(k)\text{T}}\ \ldots
\mathbf{y}_m^{(k)\text{T}}]\tran$.
Moreover, since $\mathcal{Y}$ is a polyhedron,
$\texttt{aff}(N_{\mathcal{Y}}(\mathbf{y}_i^{(k)}))=\R^n$
for all $i\in\mathcal{S}$. Thus, by the fourth result of
Lemma~\ref{lemma:normal-cone-cartesian-product-results},
$\texttt{aff}(N_{\mathcal{Y}^m}(\mathbf{y}^{(k)}))=\R^{mn}$. Hence, $\mathbf{y}^{(k)}$ satisfies the condition in
Remark~\ref{remark:condition-for-choosing-x-bar} with
$\mathcal{X}=\mathcal{Y}^m$. Furthermore, since
$\boldsymbol{\tilde\eta}_i\in
\texttt{int}~N_{\mathcal{Y}}(\mathbf{y}_i^{(k)})$
for all $i\in\mathcal{S}$, the third result of
Lemma~\ref{lemma:normal-cone-cartesian-product-results} gives
$\boldsymbol{\tilde\eta}\in
\texttt{int}~N_{\mathcal{Y}^m}(\mathbf{y}^{(k)})$.
Therefore, by Remark~\ref{Remark:feasibility-of-relaxed-problem},
problem~\eqref{eq:RFSGOR-feasibility-problem-relaxed-associated-with-consensus-problem}
is feasible with $\mathbf{\bar y}=\mathbf{y}^{(k)}$. Hence, using Proposition~\ref{prop:existance-of-RFSGOR-rays-during-iterations},
we can conclude that $\boldsymbol{\lambda}^{(k)}$ lies on an RFSGOR ray.
\end{proof}

\subsection{The Computation of Initial RFSGOR Rays}\label{Sec:computation-of-RFSGOR-rays}
We consider computing initial RFSGOR rays under $\ell_1$-norm and $\ell_{\infty}$-norm constraint sets. Since the constraint set $\mathcal{Y}$ in problem~\eqref{eq:consensus-problem} is considered to be a polyhedron (\cf~\eqref{eq:polyhedron-constraint-set}), the normal cone of $\mathcal{Y}$ at a vertex $\mathbf{\bar z}$ is given by
\begin{equation}\label{eq:normal-cone-polyhedron-set}
N_{\mathcal{Y}}(\mathbf{\bar z})=
\Big\{
\textstyle\sum_{i\in\mathcal{A}(\mathbf{\bar z})}
\alpha_i\mathbf{d}_i
\ \Big| \ 
\alpha_i\geq 0
\Big\},
\end{equation}
where
$\mathcal{A}(\mathbf{\bar z})=
\{i \mid \mathbf{d}_i\tran\mathbf{\bar z}=c_i\}$
is the set of active constraints at $\mathbf{\bar z}$,
$\mathbf{d}_i\tran$ is the $i$th row of $\mathbf{D}$, and $c_i$ is the
$i$th component of $\mathbf{c}$,
\cf~\eqref{eq:polyhedron-constraint-set}.
The following remarks provide the normal cones of the
$\ell_1$-norm and $\ell_{\infty}$-norm constraint sets at their vertices,
which can be derived from~\eqref{eq:normal-cone-polyhedron-set}.

\begin{remark}[$\mathcal{Y}$ is an $\ell_1$-ball]
\label{remark:normal-cone-l1-ball}
Let $\mathcal{Y}=\{\mathbf{z}\in\R^n \mid \|\mathbf{z}\|_1\leq a\}$,
where $a>0$. Then, the vertices of $\mathcal{Y}$ are given by
$\mathbf{\bar z}=\pm ae_j$, $j=1,\ldots,n$, where $e_j$ is the unit
vector in the $j$th coordinate direction. Then, $N_{\mathcal{Y}}(\pm ae_j)
=
\{\mathbf{n}\in\R^n \mid
\pm n_j\geq|n_i|,\ i\neq j\}$, where $\mathbf{n}=[n_1\ldots n_n]\tran$.
Moreover,
\begin{equation}\label{eq:interior-normal-cone-l1-ball}
\texttt{int}~N_{\mathcal{Y}}(\pm ae_j)
=
\{\mathbf{n}\in\R^n \mid
\pm n_j>|n_i|,\ i\neq j\}.
\end{equation}
\end{remark}

\begin{remark}[$\mathcal{Y}$ is an $\ell_\infty$-ball]
\label{remark:normal-cone-l-infinity-ball}
Let $\mathcal{Y}=\{\mathbf{z}\in\R^n \mid \|\mathbf{z}\|_\infty\leq a\}$,
where $a>0$. Then, the vertices of $\mathcal{Y}$ are given by
$\mathbf{\bar z}=a[t_1\ldots t_n]\tran$, where $t_i\in\{-1,1\}$.
Then, $N_{\mathcal{Y}}(\mathbf{\bar z})
=
\{\mathbf{n}\in\R^n \mid t_i n_i\geq0,\, \forall i\in\{1,\ldots,n\}\}$, where $\mathbf{n}=[n_1\ldots n_n]\tran$. Moreover,
\begin{equation}\label{eq:interior-normal-cone-l-infinity-ball}
\texttt{int}~N_{\mathcal{Y}}(\mathbf{\bar z})
=
\{\mathbf{n}\in\R^n \mid t_i n_i>0,\, \forall i\in\{1,\ldots,n\}\}.
\end{equation}
\end{remark}

Next, we provide the following result, which is useful for computing
initial RFSGOR rays associated with problem~\eqref{eq:consensus-problem}.

\begin{remark}\label{remark:eta-user-component}
Let $\mathbf{\bar y}
=[\mathbf{\bar y}_1\tran\ldots\mathbf{\bar y}_m\tran]\tran$, where
$\mathbf{\bar y}_i\in\texttt{bnd}~\mathcal{Y}$ for
$i\in\{1,\ldots,m\}$. Then, the solution
$\boldsymbol{\tilde\eta}
=[\boldsymbol{\tilde\eta}_1\tran\ldots\boldsymbol{\tilde\eta}_m\tran]\tran$
obtained by solving the equality constraints in~\eqref{eq:RFSGOR-feasibility-problem-relaxed-associated-with-consensus-problem}
has the form
\begin{equation}\label{eq:eta-user-component}
\boldsymbol{\tilde\eta}_i
=
\frac{1}{\beta}
\begin{cases}
\displaystyle
\mathbf{\bar y}_1-\mathbf{\bar y}_2
&\!\!; \quad i=1\\
\displaystyle
2\mathbf{\bar y}_i-\mathbf{\bar y}_{i-1}-\mathbf{\bar y}_{i+1}
&\!\!; \quad i=2,\ldots,m-1\\
\displaystyle
\mathbf{\bar y}_m-\mathbf{\bar y}_{m-1}
&\!\!; \quad i=m.
\end{cases}
\end{equation}
\end{remark}

\begin{proof}
The result follows directly from
$\boldsymbol{\eta}=\mathbf{C}\tran\mathbf{C}\mathbf{\bar y}/\beta$,
where $\mathbf{C}$ is given by~\eqref{eq:consenses-problem-matrix-C}.
\end{proof}
Then, the following proposition establishes that an RFSGOR ray associated
with problem~\eqref{eq:consensus-problem} can be computed when
$\mathcal{Y}$ is defined by an $\ell_1$-norm constraint,
\cf~Remark~\ref{remark:normal-cone-l1-ball}.
\begin{prop}\label{Prop:RFSGOR-ray-l-1-norm-constraint}
Suppose that each $f_i$ in~\eqref{eq:consensus-problem} is convex,
lower semicontinuous, and proper, and that
Assumption~\ref{Assumption:about-subdifferential-of-h-at-boundary-ponts}
holds with $h_0$ and $\mathcal{X}$ replaced by $f$ and
$\mathcal{Y}^m$, respectively, \cf~\eqref{eq:consensus-problem}.
Let $\mathcal{Y}$ be defined by an $\ell_1$-norm constraint,
\cf~Remark~\ref{remark:normal-cone-l1-ball}.
Let $\bar{\mathbf y}_i=as_i\mathbf{e}_j$, $i\in\mathcal{S}$,
where $\mathbf{e}_j$ is the unit vector in the $j$th coordinate
direction in $\R^n$, $j\in\{1,\ldots,n\}$, and
$s_i\in\{-1,1\}$ satisfy $s_1\ne s_2$, $s_m\ne s_{m-1}$,
and, for every $i\in\{2,\ldots,m-1\}$,
$s_i\ne s_{i-1}$ or $s_i\ne s_{i+1}$.
Then, there exists an RFSGOR ray $\hat{\mathcal{R}}(\bar{\mathbf y})$
(\cf~\eqref{eq:RFSGOR-ray-original}) associated with
problem~\eqref{eq:consensus-problem}, where
$\bar{\mathbf y}
=[\bar{\mathbf y}_1\tran\ldots\bar{\mathbf y}_m\tran]\tran$.
\end{prop}

\begin{proof}
Since $\bar{\mathbf y}_i=as_i\mathbf e_j$ for all
$i\in\mathcal{S}$, we have
$\bar{\mathbf y}_i\in\texttt{ver}~\mathcal{Y}$ for all
$i\in\mathcal{S}$, \cf~Remark~\ref{remark:normal-cone-l1-ball},
and hence
$\bar{\mathbf y}\in\texttt{ver}~\mathcal{Y}^m\subseteq
\texttt{bnd}~\mathcal{Y}^m$.
Moreover,
$\texttt{aff}(N_{\mathcal{Y}}(\bar{\mathbf y}_i))=\R^n$
for all $i\in\mathcal{S}$. Then, from the fourth result of
Lemma~\ref{lemma:normal-cone-cartesian-product-results}, it follows that
$\texttt{aff}(N_{\mathcal{Y}^m}(\bar{\mathbf y}))=\R^{mn}$.
Hence, $\bar{\mathbf y}$ satisfies the condition in
Remark~\ref{remark:condition-for-choosing-x-bar} with
$\mathcal{X}=\mathcal{Y}^m$. Next, from~\eqref{eq:eta-user-component}, we have $\tilde{\boldsymbol\eta}_i=
\begin{cases}
\displaystyle (a/\beta)(s_1-s_2)\mathbf e_j
&;\quad i=1,\\
\displaystyle (a/\beta)
(2s_i-s_{i-1}-s_{i+1})\mathbf e_j
&;\quad i=2,\ldots,m-1,\\
\displaystyle (a/\beta)(s_m-s_{m-1})\mathbf e_j
&;\quad i=m.
\end{cases}$. Since $s_1\ne s_2$ and $s_m\ne s_{m-1}$, we have
$s_1(s_1-s_2)>0$ and $s_m(s_m-s_{m-1})>0$.
Moreover, for every $i\in\{2,\ldots,m-1\}$,
the condition $s_i\ne s_{i-1}$ or $s_i\ne s_{i+1}$
implies that
$s_i(2s_i-s_{i-1}-s_{i+1})>0$.
Thus, from~\eqref{eq:interior-normal-cone-l1-ball},
it follows that
$\tilde{\boldsymbol\eta}_i\in
\texttt{int}~N_{\mathcal{Y}}(\bar{\mathbf y}_i)$
for all $i\in\mathcal{S}$.
Hence, from the third result of
Lemma~\ref{lemma:normal-cone-cartesian-product-results},
we have
$\tilde{\boldsymbol\eta}\in
\texttt{int}~N_{\mathcal{Y}^m}(\bar{\mathbf y})$.
Therefore, according to
Lemma~\ref{lemma:feasibility-of-relaxed-problem-implies-feasibility-of-original-problem}
and Remark~\ref{Remark:feasibility-of-relaxed-problem},
an RFSGOR ray $\hat{\mathcal{R}}(\bar{\mathbf y})$
associated with problem~\eqref{eq:consensus-problem}
can be computed.
\end{proof}
Next, the following result establishes that an RFSGOR ray associated with
problem~\eqref{eq:consensus-problem} can be computed when $\mathcal{Y}$ is
defined by an $\ell_\infty$-norm constraint,
\cf~Remark~\ref{remark:normal-cone-l-infinity-ball}.
\begin{prop}\label{Prop:RFSGOR-ray-l-infinity-norm-constraint}
Suppose that each $f_i$ in~\eqref{eq:consensus-problem} is convex,
lower semicontinuous, and proper, and that
Assumption~\ref{Assumption:about-subdifferential-of-h-at-boundary-ponts}
holds with $h_0$ and $\mathcal{X}$ replaced by $f$ and
$\mathcal{Y}^m$, respectively, \cf~\eqref{eq:consensus-problem}. Let $\mathcal{Y}$ be defined by an $\ell_\infty$-norm constraint,
\cf~Remark~\ref{remark:normal-cone-l-infinity-ball}.
Let
\begin{equation}\label{eq:consensus-problem-feasibile-y-bar-l-infinity-norm}
\bar{\mathbf{y}}_i=(-1)^{i-1}a\mathbf{1}_n,
\quad i\in\mathcal{S}.
\end{equation}
Then, there exists an RFSGOR ray $\hat{\mathcal{R}}(\mathbf{\bar y})$
(\cf~\eqref{eq:RFSGOR-ray-original}) associated with problem
\eqref{eq:consensus-problem}, where
$\bar{\mathbf{y}}=[\bar{\mathbf{y}}_1\tran\ldots\bar{\mathbf{y}}_m\tran]\tran$.
\end{prop}
\begin{proof}
Clearly, $\mathbf{\bar y}_i\in\texttt{ver}~\mathcal{Y}$ for all
$i\in\mathcal{S}$, \cf~Remark~\ref{remark:normal-cone-l-infinity-ball}.
Then, similarly to the proof of
Proposition~\ref{Prop:RFSGOR-ray-l-1-norm-constraint},
$\mathbf{\bar y}$ satisfies the condition in
Remark~\ref{remark:condition-for-choosing-x-bar} with
$\mathcal{X}=\mathcal{Y}^m$. From~\eqref{eq:eta-user-component}, the solution
$\boldsymbol{\tilde\eta}
=[\boldsymbol{\tilde\eta}_1\tran\ldots\boldsymbol{\tilde\eta}_m\tran]\tran$
obtained by solving the equality constraints in
\eqref{eq:RFSGOR-feasibility-problem-relaxed-associated-with-consensus-problem}
is given by
\begin{equation}\label{eq:l-infinity-norm-eta}
\boldsymbol{\tilde\eta}_i=
\begin{cases}
\displaystyle (1/\beta)2a\mathbf{1}_n
&\!\!;\quad i=1\\
\displaystyle (1/\beta)(-1)^{i-1}4a\mathbf{1}_n
&\!\!;\quad i=2,\ldots,m-1\\
\displaystyle (1/\beta)(-1)^{m-1}2a\mathbf{1}_n
&\!\!;\quad i=m.
\end{cases}
\end{equation}
Then, from~\eqref{eq:interior-normal-cone-l-infinity-ball},
$\boldsymbol{\tilde\eta}_i\in
\texttt{int}~N_{\mathcal{Y}}(\mathbf{\bar y}_i)$
for all $i\in\mathcal{S}$.
The remainder of the proof follows the same arguments as the proof of
Proposition~\ref{Prop:RFSGOR-ray-l-1-norm-constraint}.
\end{proof}

\section{Communication Efficiency of SSAR}
\label{Sec:Communication-Efficiency-of-AdFSGOR}
In this section, we discuss the benefits of the SSAR method from a communication perspective. To this end, we first consider the total number of bits transmitted over the iterations of Algorithm~\ref{Alg:AdFSGOR-assisted_dual-decomposition-method}, accounting for both agent-to-CS and CS-to-agent communication. We refer to communication from the agents to the CS as uplink communication and communication from the CS to the agents as downlink communication. In uplink communication, the agents transmit $1$ bit (\cf~steps~$4$ and $10$), the vertex indices corresponding to $\mathbf{y}_i^{(k)}$ (\cf~step~$17$), or the numerical values of $\mathbf{y}_i^{(k)}$ (\cf~steps~$13$ and $28$). In downlink communication, the CS broadcasts $1$ bit (\cf~steps~$5$ and $11$), the vertex indices corresponding to $\mathbf{y}_i^{(k)}$ (\cf~step~$18$), or the numerical values of $\boldsymbol{\lambda}^{(k+1)}$ (\cf~steps~$15$ and $28$).

Since $\mathcal{Y}$ is bounded, each component of $\mathbf{y}_i^{(k)}$ lies in a bounded interval. Therefore, we assume that each component of $\mathbf{y}_i^{(k)}$ is represented using a $b_y$-bit uniform quantizer (see Figure~\ref{Fig:quantization-scheme} for an illustration). Moreover, since $\mathcal{Y}$ has a finite number of vertices, assuming that these vertices are indexed and known to both the agents and the CS, transmitting the index of a vertex requires $b_{\mathrm{v}}=\left\lceil\log_2|\texttt{ver}~\mathcal{Y}|\right\rceil$ bits, where $|\texttt{ver}~\mathcal{Y}|$ denotes the total number of vertices of $\mathcal{Y}$ and, for any $x\in\R$, $\lceil x\rceil$ denotes the smallest integer $n_0$ such that $n_0\geq x$. Note that the dual problem~\eqref{eq:dual-problem-of-consensus-problem} is unconstrained, and hence the dual variable is not restricted to a bounded set. Therefore, instead of using a uniform quantizer, we assume that each component of $\boldsymbol{\lambda}^{(k+1)}$ is communicated using a $b_\lambda$-bit floating-point representation.

\begin{figure}[t]
\centering
\begin{tikzpicture}[scale=1.0]

\draw[]  (-4,0) -- (4.4,0);

\node at (-4,0.4) {$-a$};
\node at (4.4,0.4) {$a$};

\foreach \x/\i in {-4/0, -2.8/1, -1.6/2, -0.4/3, 0.8/4, 2/5, 3.2/6, 4.4/7} {
    \draw (\x,0.1) -- (\x,-0.1);
    \node at (\x,-0.5) {\small $\i$};
}

\filldraw[red] (0.4,0) circle (2pt);
\node[above] at (0.3,0.01) {$y_{ij}^{(k)}$};

\filldraw[blue] (0.8,0) circle (2pt);
\node[above] at (1.2,0.01) {$y_{ij}^{q(k)}$};

\draw[<->] (-2.8,0.2) -- (-1.6,0.2);
\node at (-2.2,0.5) {$t=2a/7$};

\end{tikzpicture}
\caption{Illustration of uniform $b_y$-bit quantization over an interval $[-a,a]$ with $b_y=3$, where $a>0$. The corresponding number of quantization levels is $L=2^3=8$. Here, $y_{ij}^{(k)}$ represents the $j$th component of $\mathbf{y}_i^{(k)}$, and $y_{ij}^{q(k)}$ represents the corresponding quantized value.}
\label{Fig:quantization-scheme}
\end{figure}

\begin{figure}[!t]
\centering

\begin{tikzpicture}[x=0.13cm,y=0.13cm]
\node[left, font=\footnotesize] at (0,22) {Alg.~1};

\draw[fill=green!30]
    (0,20) rectangle (50,24);
\draw[->, thick] (0,20) -- (52,20)
    node[below, font=\scriptsize] {$k$};
\draw[->, thick] (0,20) -- (0,26)
    node[above, font=\scriptsize] {$b$};

 \node[above, font=\scriptsize] at (25,24) {65600};
x-axis markings
\foreach \x in {50}
{
    \draw[thick] (\x,20) -- (\x,19.5);
}

 \node[below, font=\scriptsize] at (49,20) {50};
\node[left, font=\footnotesize] at (-1,12) {\cite{abeynanda2026characteristics}};

\draw[fill=blue!30]
    (0,10) rectangle (20,11);

\draw[fill=blue!30]
    (20,10) rectangle (21,15);

\draw[fill=green!30]
    (21,10) rectangle (50,14);

\node[above, font=\scriptsize] at (10,11) {11};
\node[above, font=\scriptsize] at (35,14) {65600};
\node[above, font=\scriptsize] at (20.5,15) {65611};

\draw[->, thick] (0,10) -- (52,10)
    node[below, font=\scriptsize] {$k$};
\draw[->, thick] (0,10) -- (0,16)
    node[above, font=\scriptsize] {$b$};

x-axis markings
\foreach \x in {20,50}
{
    \draw[thick] (\x,10) -- (\x,9.5);
}

\node[below, font=\scriptsize] at (20,10) {20};
\node[below, font=\scriptsize] at (49,10) {50};
\node[left, align=right, font=\footnotesize] at (0,2)
    {Alg.~3};

\draw[fill=red!30]
    (0,0) rectangle (20,1);

\draw[fill=red!30]
    (20,0) rectangle (21,3);

\draw[fill=red!30]
    (21,0) rectangle (30,1);

\draw[fill=red!30]
    (30,0) rectangle (31,3);

\draw[fill=red!30]
    (31,0) rectangle (35,1);

\draw[fill=red!30]
    (35,0) rectangle (36,6);

\draw[fill=green!30]
    (36,0) rectangle (50,4);

\node[above, font=\scriptsize] at (10,1) {11};
\node[above, font=\scriptsize] at (20.5,3) {182};
\node[above, font=\scriptsize] at (25.5,1) {11};
\node[above, font=\scriptsize] at (30.5,3) {182};
\node[above, font=\scriptsize] at (33,1) {11};
\node[above, font=\scriptsize] at (35.5,6) {65622};
\node[above, font=\scriptsize] at (43,4) {65600};

\draw[->, thick] (0,0) -- (52,0)
    node[below, font=\scriptsize] {$k$};
\draw[->, thick] (0,0) -- (0,5.8)
    node[above, font=\scriptsize] {$b$};

\foreach \x in {20,30,35,50}
{
    \draw[thick] (\x,0) -- (\x,-0.5);
}

\node[below, font=\scriptsize] at (0,0) {0};
\node[below, font=\scriptsize] at (20,0) {20};
\node[below, font=\scriptsize] at (30,0) {30};
\node[below, font=\scriptsize] at (35,0) {35};
\node[below, font=\scriptsize] at (49,0) {50};

\end{tikzpicture}
\caption{Communication efficiency of SSAR: communication cost
at each iteration of Alg.~\ref{Alg:AdFSGOR-assisted_dual-decomposition-method}, \cite[Alg. 3]{abeynanda2026characteristics}, and Alg.~\ref{Alg:Dual-Decomposition-Algorithm}. Numbers above the rectangles indicate their heights.}
\label{Fig:comparison_communication_efficiency}
\end{figure}

Next, we quantify the total communication cost of Algorithm~\ref{Alg:AdFSGOR-assisted_dual-decomposition-method}. An important feature of Algorithm~\ref{Alg:AdFSGOR-assisted_dual-decomposition-method} is that, throughout the SSAR exploitation phase, the agents and the CS do not need to communicate the numerical values of the local solutions $\mathbf{y}_i^{(k)}$ or the dual variable updates $\boldsymbol{\lambda}^{(k+1)}$. Instead, only binary indicators and, at iterations at which ray transitions occur, the vertex indices corresponding to $\mathbf{y}_i^{(k)}$ need to be communicated. In particular, at iterations without ray transitions, each agent only needs to transmit $1$ bit to the CS, and the CS only needs to broadcast $1$ bit to the agents, \cf~steps~$3$--$8$. At iterations at which ray transitions occur, each agent needs to transmit $2+b_{\mathrm{v}}$ bits to the CS (\cf~steps~$4$, $10$, and $17$), and the CS needs to broadcast $2+mb_{\mathrm{v}}$ bits to the agents (\cf~steps~$5$, $11$, and $18$).

Let $k_0$ denote the iteration at which the SSAR exploitation phase terminates. The SSAR exploitation phase can terminate at step~$12$ or step~$23$ of Algorithm~\ref{Alg:AdFSGOR-assisted_dual-decomposition-method}. In the first case, when $k=k_0$, each agent transmits $2+nb_y$ bits to the CS (\cf~steps~$4$, $10$, and $13$), and the CS broadcasts $2+n(m-1)b_\lambda$ bits to the agents (\cf~steps~$5$, $11$, and $15$). In the latter case, when $k=k_0$, each agent transmits $2+b_{\mathrm v}$ bits to the CS (\cf~steps~$4$, $10$, and $17$), and the CS broadcasts $2+mb_{\mathrm v}$ bits to the agents (\cf~steps~$5$, $11$, and $18$). After switching to Algorithm~\ref{Alg:Dual-Decomposition-Algorithm}, each agent transmits $nb_y$ bits to the CS, and the CS broadcasts $n(m-1)b_\lambda$ bits to the agents at each iteration (\cf~steps~$4$ and $6$ of Algorithm~\ref{Alg:Dual-Decomposition-Algorithm}). Based on the communication structure discussed above, the total number of uplink communication bits in Algorithm~\ref{Alg:AdFSGOR-assisted_dual-decomposition-method} at iteration $k$ is given by

\begin{equation}\label{eq:bits-uplink-SSAR}
B_{\texttt{SSAR,U}}^{(k)}
{=}
\begin{cases}
m & \!\! \!\!\! ; \ k<k_{0} \, \text{\& no ray transition occurs}
\\
(2{+}b_v)m &\!\! \!\!\!; \ k<k_{0} \,\text{\& a ray transition occurs}
\\
(2{+}nb_y)m &\!\! \!\!\! ; \ k=k_{0}\,\text{\& step~$12$ of Alg.~\ref{Alg:AdFSGOR-assisted_dual-decomposition-method} occurs}
\\
(2{+}b_v)m &\!\! \!\!\!; \ 
k=k_{0}\,\text{\& step~$22$ of Alg.~\ref{Alg:AdFSGOR-assisted_dual-decomposition-method} occurs}
\\
mnb_y &\!\! \!\!\! ; \ k>k_{0}.
\end{cases}
\end{equation}
The total number of downlink communication bits in Algorithm~\ref{Alg:AdFSGOR-assisted_dual-decomposition-method} at iteration $k$ is given by
\begin{equation}\label{eq:bits-downlink-SSAR}
B_{\texttt{SSAR,D}}^{(k)}
{=}
\begin{cases}
1 \ ; \ k<k_{0} \, \text{\& no ray transition occurs}
\\
2+mb_v \ ; \ k<k_{0} \,\text{\& a ray transition occurs}
\\
2{+}n(m{-}1)b_\lambda  ; \, k{=}k_{0}\,\text{\& step $12$ of Alg.~\ref{Alg:AdFSGOR-assisted_dual-decomposition-method} occurs}
\\
2+mb_v \ ; \ k=k_{0}\,\text{\& step~$22$ of Alg.~\ref{Alg:AdFSGOR-assisted_dual-decomposition-method} occurs}\\
n(m-1)b_\lambda \ ; \ k>k_{0}.
\end{cases}
\end{equation}
Then, the total number of bits communicated over the network in Algorithm~\ref{Alg:AdFSGOR-assisted_dual-decomposition-method} at iteration $k$ is given by
\begin{equation}\label{eq:total-bits-SSAR}
B_{\texttt{SSAR}}^{(k)}
=
B_{\texttt{SSAR,U}}^{(k)}
+
B_{\texttt{SSAR,D}}^{(k)}.
\end{equation}

Next, to compare the communication cost of the SSAR algorithm (Algorithm~\ref{Alg:AdFSGOR-assisted_dual-decomposition-method}) with Algorithm~\ref{Alg:Dual-Decomposition-Algorithm} and the FSGOR method~\cite[Alg. 3]{abeynanda2026characteristics}, we first quantify the total communication cost of Algorithm~\ref{Alg:Dual-Decomposition-Algorithm} and \cite[Alg. 3]{abeynanda2026characteristics}. Note that Algorithm~\ref{Alg:Dual-Decomposition-Algorithm} requires each agent to transmit $\mathbf{y}_i^{(k)}$ to the CS (\cf~step~$4$), and the CS to broadcast $\boldsymbol{\lambda}^{(k+1)}$ to the agents (\cf~step~$6$) at each iteration $k$. Thus, Algorithm~\ref{Alg:Dual-Decomposition-Algorithm} requires $B_{\texttt{Alg}1}^{(k)}=mnb_y+n(m-1)b_{\lambda}$ bits to be communicated over the network at each iteration $k$. Next, we determine the total communication cost of the FSGOR method~\cite[Alg. 3]{abeynanda2026characteristics}. The FSGOR method identifies only an initial RFSGOR ray $\mathcal{R}_1$. Suppose that it detects that $\boldsymbol{\lambda}^{(k)}$ exits its initial RFSGOR ray at iteration $k=t_0$. Then, the total number of bits communicated in \cite[Alg. 3]{abeynanda2026characteristics} at iteration $k$ is given by
\begin{equation}\label{eq:total-bits-FSGOR}
B_{\texttt{FSGOR}}^{(k)}
=
\begin{cases}
m+1 &\!\!\!\!;\quad k<t_0,
\\
m+1+mnb_y+n(m-1)b_{\lambda} &\!\!\!\!;\quad k=t_0,
\\
mnb_y+n(m-1)b_{\lambda} &\!\!\!\!;\quad k>t_0.
\end{cases}
\end{equation}
Next, we illustrate the communication efficiency of Algorithm~\ref{Alg:AdFSGOR-assisted_dual-decomposition-method} compared with Algorithm~\ref{Alg:Dual-Decomposition-Algorithm} and the FSGOR method~\cite[Alg. 3]{abeynanda2026characteristics} using the following example.

\begin{example}[Communication Efficiency of Algorithm~\ref{Alg:AdFSGOR-assisted_dual-decomposition-method}]
Consider problem~\eqref{eq:consensus-problem} with $m=10$, $n=100$, and $\mathcal{Y}$ defined by an $l_1$-ball (\cf~Remark~\ref{remark:normal-cone-l1-ball}). Let $b_y=8$ and $b_\lambda=64$. Since $\mathcal{Y}$ is an $l_1$-ball, it has $2n=200$ vertices, and hence $b_{\mathrm{v}}=\lceil\log_2(200)\rceil=8$. Suppose that Algorithm~\ref{Alg:AdFSGOR-assisted_dual-decomposition-method} identifies three RFSGOR rays, $\mathcal{R}_1$, $\mathcal{R}_2$, and $\mathcal{R}_3$, where $\boldsymbol{\lambda}^{(k)}\in\mathcal{R}_1$ for $20$ iterations, $\boldsymbol{\lambda}^{(k)}\in\mathcal{R}_2$ for $10$ iterations, and $\boldsymbol{\lambda}^{(k)}\in\mathcal{R}_3$ for $5$ iterations. Furthermore, suppose that the SSAR exploitation phase terminates due to step~$12$ of Algorithm~\ref{Alg:AdFSGOR-assisted_dual-decomposition-method}. For comparison, suppose that we initialize both Algorithm~\ref{Alg:AdFSGOR-assisted_dual-decomposition-method} and the FSGOR method~\cite[Alg. 3]{abeynanda2026characteristics} from the same point on the same initial RFSGOR ray $\mathcal{R}_1$ and use the same constant stepsize while $\boldsymbol{\lambda}^{(k)}$ lies on $\mathcal{R}_1$. Therefore, the FSGOR method also remains on $\mathcal{R}_1$ for $20$ iterations. Let $T_{\texttt{SSAR}}^{50}$, $T_{\texttt{FSGOR}}^{50}$, and $T_{\texttt{Alg1}}^{50}$ be the total numbers of bits communicated over the network over $50$ iterations for Algorithm~\ref{Alg:AdFSGOR-assisted_dual-decomposition-method}, the FSGOR method~\cite[Alg. 3]{abeynanda2026characteristics}, and Algorithm~\ref{Alg:Dual-Decomposition-Algorithm}, respectively. Then, we have $T_{\texttt{SSAR}}^{50}=\sum_{k=1}^{50}(B_{\texttt{SSAR,U}}^{(k)}+B_{\texttt{SSAR,D}}^{(k)})=984749$, \cf~\eqref{eq:bits-uplink-SSAR}, \eqref{eq:bits-downlink-SSAR}, and \eqref{eq:total-bits-SSAR}. Next, from \eqref{eq:total-bits-FSGOR}, we have $T_{\texttt{FSGOR}}^{50}=\sum_{k=1}^{50}B_{\texttt{FSGOR}}^{(k)}=1968231$. Finally, $T_{\texttt{Alg1}}^{50}=50(mnb_y+n(m-1)b_{\lambda})=3280000$. Thus, Algorithm~\ref{Alg:AdFSGOR-assisted_dual-decomposition-method} can save $983482$ bits and $2295251$ bits compared with the FSGOR method~\cite[Alg. 3]{abeynanda2026characteristics} and Algorithm~\ref{Alg:Dual-Decomposition-Algorithm}, respectively. An illustration is given in Fig.~\ref{Fig:comparison_communication_efficiency}.
\end{example}

\section{Numerical Results}\label{sec:numerical-results}
In this section, we numerically compare the performance of the proposed Algorithm~\ref{Alg:AdFSGOR-assisted_dual-decomposition-method} with Algorithm~\ref{Alg:Dual-Decomposition-Algorithm} and the FSGOR method~\cite[Alg. 3]{abeynanda2026characteristics}. Algorithm~\ref{Alg:Dual-Decomposition-Algorithm} is considered with the following stepsize rules: the constant stepsize rule, the polynomially decaying stepsize rule $\gamma_0/k$, where $\gamma_0>0$ is chosen appropriately, the step-decay stepsize rule~\cite{wang_2021}, and the stepsize rule~\cite{Yura_2020}. The numerical comparisons are performed in terms of both convergence behavior and communication efficiency. The results are presented under two problem settings: 1) least squares regression (LSR) and 2) logistic regression, both with an $\ell_1$-norm regularizer. First, note that both problem settings are based on the following regularized problem formulation:
\begin{equation}\label{eq:regularized-problem}
    \begin{array}{ll}
    \mbox{minimize} & f(\mathbf{z})=(1/m)\sum_{i{=}1}^{m}f_i(\mathbf{z})
    +u\|\mathbf{z}\|_1,
    \end{array}
\end{equation}
where $f_i : \R^n \to \R$ is a convex local objective
function associated with agent $i$, $i\in\mathcal{S}$, and $u$ is the
regularization parameter. Then, the regularized problem
\eqref{eq:regularized-problem} can be equivalently reformulated as the
following regularization-constrained problem (\cf~\cite{bach_2012}):
\begin{equation}\label{eq:regularization-constrained-problem}
    \begin{array}{ll}
    \mbox{minimize} &\tilde f(\mathbf{z})=(1/m)\sum_{i{=}1}^{m}f_i(\mathbf{z})\\
    \mbox{subject to} & \|\mathbf{z}\|_1\leq \tilde u,
    \end{array}
\end{equation}
where $\tilde u>0$ is the corresponding regularization-constrained
bound for a given regularization parameter $u$ in \eqref{eq:regularized-problem}. Then, we consider the following
equivalent consensus formulation of problem
\eqref{eq:regularization-constrained-problem}: 
\begin{equation} \label{eq:regularization-constrained-consensus-form}
\begin{array}{ll}
\mbox{minimize} & f(\mathbf{y})=(1/m)\sum_{i{=}1}^{m}f_i(\mathbf{y}_{i}) \\
\mbox{subject to} & \|\mathbf{y}_i\|_1\leq \tilde u,\ i\in \mathcal{S}\\
& \mathbf{C}\mathbf{y}=\mathbf{0}_{n(m-1)},
\end{array}
\end{equation}
where $\mathbf{y}=[\mathbf{y}_1\tran \ \ldots \ \mathbf{y}_m\tran]\tran$ and
$\mathbf{C}$ has the form~\eqref{eq:consenses-problem-matrix-C}.

\subsection{Least Squares Regression with Regularization}\label{subsec:numerical-results-least-squares-Regression}
We consider problem~\eqref{eq:regularization-constrained-consensus-form}
as a LSR problem, where
$f_i(\mathbf{y}_i)=(1/N_i)\sum_{j=1}^{N_i}
(\mathbf{y}_i\tran\mathbf{x}_{ij}-t_{ij})^2$.
Here, $\mathcal{D}_i=\{(\mathbf{x}_{ij},t_{ij})\}_{j=1}^{N_i}$ denotes
the training dataset of agent $i$, where $\mathbf{x}_{ij}\in\R^n$ is the
$j$th feature vector, whose first entry is $1$ for the bias term, and
$t_{ij}\in\R$ is the corresponding label associated with agent $i$.
The experiments are conducted using the Year Prediction Million Song
Dataset (MSD)~\cite{year_prediction_msd}, which contains $515345$
data examples with $\mathbf{x}_{ij}\in\R^{91}$. The first $463715$
data examples are considered for training. This training dataset is
standardized and shuffled before running the experiments. We consider
$m=100$ agents, with the first $99$ agents having $4637$ training data
examples each, i.e., $N_i=4637$ for $i=1,\ldots,99$, and the $100$th
agent having $N_{100}=4652$ training data examples.

\begin{figure}[!t]
\centering
\subfloat[]{\includegraphics[width=0.5\linewidth]%
{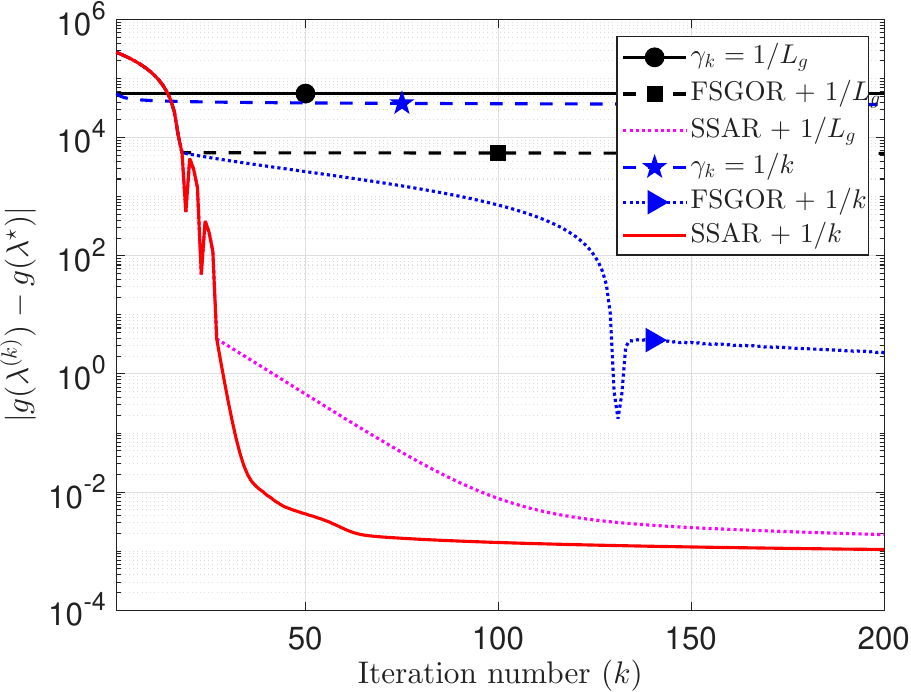}%
\label{subfig:comparison-LSR-1}}
\hfil
\subfloat[]{\includegraphics[width=0.5\linewidth]%
{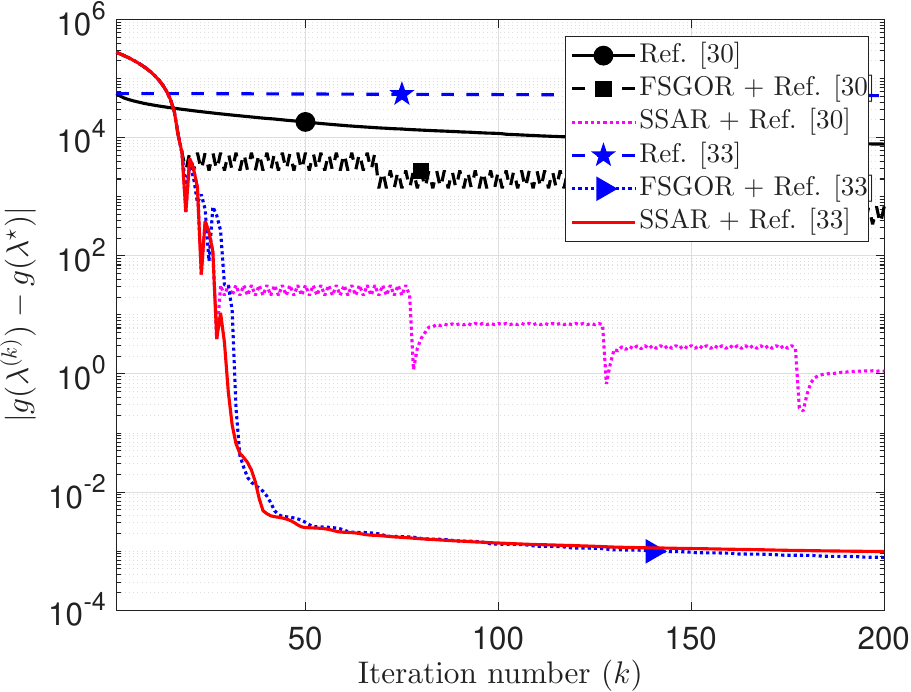}%
\label{subfig:comparison-LSR-2}}
\hfil
\caption{LSR with MSD: comparison with Alg.~\ref{Alg:Dual-Decomposition-Algorithm} and FSGOR~\cite[Alg. 3]{abeynanda2026characteristics}. \protect\subref{subfig:comparison-LSR-1} Alg.~\ref{Alg:Dual-Decomposition-Algorithm} used traditional stepsizes. 
\protect\subref{subfig:comparison-LSR-2} Alg.~\ref{Alg:Dual-Decomposition-Algorithm} used step decay stepsize~\cite{wang_2021} and stepsize \cite{Yura_2020}.}  
\label{Fig:comparison-LSR}
\end{figure}

\begin{figure}[!t]
\centering
\subfloat[]{\includegraphics[width=0.5\linewidth]%
{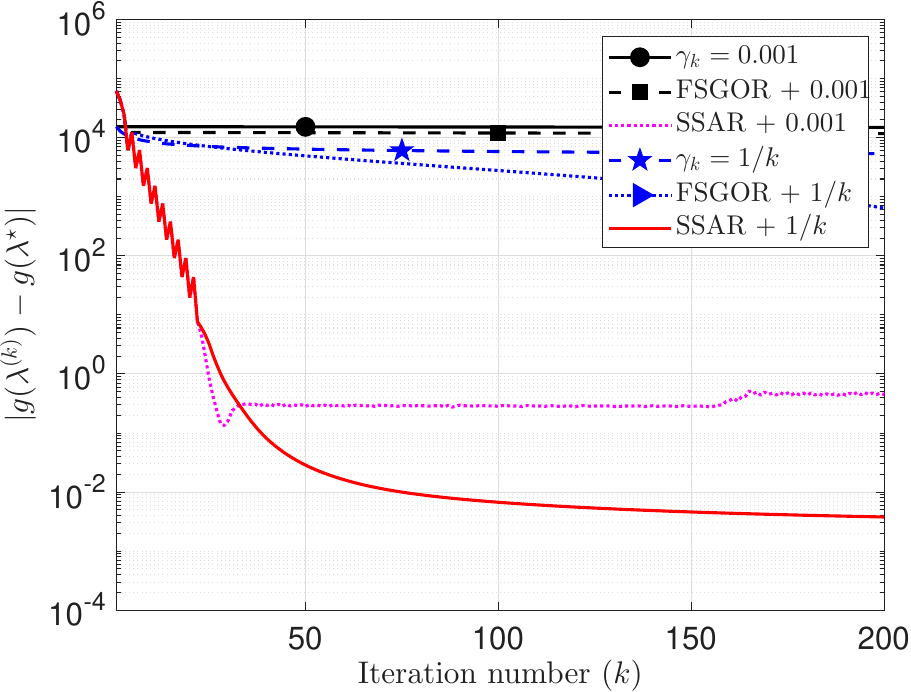}%
\label{subfig:comparison-logistic-regresssion-1}}
\hfil
\subfloat[]{\includegraphics[width=0.5\linewidth]%
{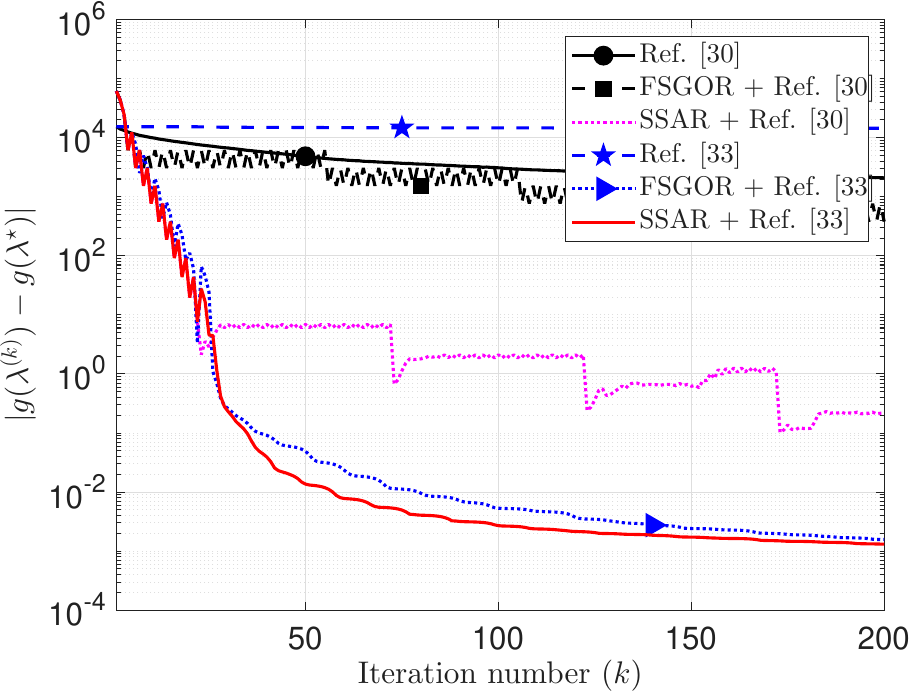}%
\label{subfig:comparison-logistic-regresssion-2}}
\hfil
\caption{Logistic regression with Spambase dataset: comparison with Alg.~\ref{Alg:Dual-Decomposition-Algorithm} and FSGOR~\cite[Alg. 3]{abeynanda2026characteristics}. \protect\subref{subfig:comparison-logistic-regresssion-1} Alg.~\ref{Alg:Dual-Decomposition-Algorithm} used traditional stepsizes. 
\protect\subref{subfig:comparison-logistic-regresssion-2} Alg.~\ref{Alg:Dual-Decomposition-Algorithm} used step decay stepsize~\cite{wang_2021} and stepsize~\cite{Yura_2020}.}  
\label{Fig:comparison-logistic-regresssion}
\end{figure}

Fig.~\ref{Fig:comparison-LSR} depicts the comparison of the results, where
the parameters of the SSAR algorithm
(Algorithm~\ref{Alg:AdFSGOR-assisted_dual-decomposition-method})
are set to $(\beta,c,\gamma,q,N_{\max})=(5,500,6,0.3,8)$. The FSGOR method
~\cite[Alg. 3]{abeynanda2026characteristics} is considered with the same
values $(\beta,c,\gamma)=(5,500,6)$ as those used for
Algorithm~\ref{Alg:AdFSGOR-assisted_dual-decomposition-method}.
The constant stepsize rule used in
Algorithm~\ref{Alg:Dual-Decomposition-Algorithm} is set to
$\gamma_k=1/L_g$, where $L_g$ is the gradient Lipschitz constant of the dual function associated with the considered LSR problem. Fig.~\ref{Fig:comparison-LSR}
\subref{subfig:comparison-LSR-1} shows that
Algorithm~\ref{Alg:AdFSGOR-assisted_dual-decomposition-method}
converges much faster than both
Algorithm~\ref{Alg:Dual-Decomposition-Algorithm} and \cite[Alg. 3]{abeynanda2026characteristics}. Fig.~\ref{Fig:comparison-LSR}
\subref{subfig:comparison-LSR-2} demonstrates that
Algorithm~\ref{Alg:AdFSGOR-assisted_dual-decomposition-method}
achieves significantly faster convergence than
Algorithm~\ref{Alg:Dual-Decomposition-Algorithm}. Compared with the
FSGOR method~\cite[Alg. 3]{abeynanda2026characteristics},
Algorithm~\ref{Alg:AdFSGOR-assisted_dual-decomposition-method}
exhibits faster convergence when the switching stepsize rule
~\cite{wang_2021} is used, while achieving comparable convergence when
the switching stepsize rule~\cite{Yura_2020} is used.

\subsection{Logistic Regression with Regularization}\label{subsec:neumerical-results-logistic-regression}
We consider problem~\eqref{eq:regularization-constrained-consensus-form} as a logistic regression problem, where
$f_i(\mathbf{y}_i)=(1/N_i)\sum_{j=1}^{N_i}
[\log(1+e^{\mathbf{y}_i\tran\mathbf{x}_{ij}})
-t_{ij}\mathbf{y}_i\tran\mathbf{x}_{ij}]$.
Here, $\mathcal{D}_i={(\mathbf{x}_{ij},t_{ij})}_{j=1}^{N_i}$ is the training dataset of agent $i$, where $\mathbf{x}_{ij}\in\R^n$ is the $j$th feature vector containing $1$ as the first entry for the bias term, and $t_{ij}\in\{0,1\}$ is the corresponding binary outcome associated with agent $i$. The experiments are conducted using the Spambase dataset~\cite{spambase_94}. The Spambase dataset contains $4601$ data examples with $\mathbf{x}_{ij}\in\R^{58}$. A total of $3220$ data examples are randomly selected for training and divided among $m=10$ agents, with $322$ training data examples being assigned to each agent, i.e., $N_i=322$ for all $i\in\mathcal{S}$. The feature matrix is standardized before the experiments are conducted.

Fig.~\ref{Fig:comparison-logistic-regresssion} shows the corresponding convergence results. The parameters of the SSAR algorithm (\cf~Algorithm~\ref{Alg:AdFSGOR-assisted_dual-decomposition-method}) are chosen as $(\beta,c,\gamma,q,N_{\max})=(5,500,6,0.5,13)$. For the FSGOR method~\cite[Alg. 3]{abeynanda2026characteristics}, the same values $(\beta,c,\gamma)=(5,500,6)$ are considered. The results demonstrate a similar behavior to that observed in Fig.~\ref{Fig:comparison-LSR}.

\subsection{Communication Efficiency of SSAR}
We evaluate the communication efficiency of SSAR
(\cf~Algorithm~\ref{Alg:AdFSGOR-assisted_dual-decomposition-method})
for the LSR and logistic regression problems
considered in \S~\ref{subsec:numerical-results-least-squares-Regression}
and \S~\ref{subsec:neumerical-results-logistic-regression},
respectively. We use the communication schemes described
in \S~\ref{Sec:Communication-Efficiency-of-AdFSGOR},
with $b_y=8$ and $b_\lambda=64$ for both problems.
Since the equivalent regularization-constrained formulations
of both problems have $\ell_1$-ball constraint sets
(\cf~\eqref{eq:regularization-constrained-consensus-form}),
we have $b_v=\lceil\log_2(2n)\rceil$, where $n=91$
and $n=58$ for the LSR and logistic regression problems,
respectively.

For the LSR problem, we observed that SSAR identified
$5$ RFSGOR rays $\mathcal{R}_1,\ldots,\mathcal{R}_5$,
where $\mathcal{R}_{j+1}=-c_j\mathcal{R}_j$ for
$j\in\{1,\ldots,4\}$ and $c_j\in\R^+$, and that
the SSAR exploitation phase ended at $k=k_0=27$.
Fig.~\ref{Fig:comminication-efficiency-of-AdFSGOR}
\subref{subfig:comminication-efficiency-of-AdFSGOR-linear-regresssion}
shows the number of bits transmitted at each iteration for SSAR (bottom), the FSGOR method
~\cite[Alg. 3]{abeynanda2026characteristics} (middle),
and Algorithm~\ref{Alg:Dual-Decomposition-Algorithm} (top).
Over $40$ iterations, the communication costs are
$T_{\texttt{SSAR}}^{40}=9100896$,
$T_{\texttt{FSGOR}}^{40}=14937466$, and
$T_{\texttt{Alg1}}^{40}=25975040$. For logistic regression, we observed that SSAR identified
$10$ RFSGOR rays $\mathcal{R}_1,\ldots,\mathcal{R}_{10}$,
where $\mathcal{R}_{j+1}=-c_j\mathcal{R}_j$ for
$j\in\{1,\ldots,9\}$ and $c_j\in\R^+$, and that
the SSAR exploitation phase ended at $k=k_0=22$.
Fig.~\ref{Fig:comminication-efficiency-of-AdFSGOR}
\subref{subfig:comminication-efficiency-of-AdFSGOR-logistic-regresssion}
shows the corresponding communication costs at each iteration. Over $40$ iterations, we have
$T_{\texttt{SSAR}}^{40}=724524$,
$T_{\texttt{FSGOR}}^{40}=1369783$, and
$T_{\texttt{Alg1}}^{40}=1521920$.
Thus, in both problems, SSAR considerably reduces
communication costs compared with FSGOR and
Algorithm~\ref{Alg:Dual-Decomposition-Algorithm}.

To further evaluate communication efficiency,
Fig.~\ref{Fig:comminication-efficiency-of-AdFSGOR}
\subref{subfig:error-vs-bits-linear-regresssion}
and Fig.~\ref{Fig:comminication-efficiency-of-AdFSGOR}\subref{subfig:error-vs-bits-logistic-regresssion}
show the average absolute dual objective error
$\bar E(B)$ versus the number of transmitted bits
$B$ for the LSR and logistic regression problems,
respectively, averaged over $10$ realizations.
Here, $\bar E(B){=}(1/10)\sum_{s=1}^{10}
|g_s(\boldsymbol{\lambda}_s^{(k_s(B))})
{-}g_s(\boldsymbol{\lambda}_s^\star)|$,
where $g_s$ is the dual function,
$\boldsymbol{\lambda}_s^\star$ is a dual optimal
solution, and $k_s(B)$ is the last iteration
completed within budget $B$, all for realization $s$.
For each realization, the training data and
initialization points were randomized.
Specifically, $309000$ samples were randomly selected
from the first $463715$ samples of the MSD, while
$3220$ samples were randomly selected from the
$4601$ samples of the Spambase dataset.
Both figures show that SSAR achieves a given accuracy
with considerably fewer transmitted bits than FSGOR
and Algorithm~\ref{Alg:Dual-Decomposition-Algorithm},
while attaining a lower error for a given
communication budget.

\begin{figure}[!t]
\centering
\subfloat[]{%
\begin{tikzpicture}[x=0.095cm,y=0.115cm]

\draw[fill=green!30]
    (0,20) rectangle (40,26);
\draw[->, thick] (0,20) -- (42,20);
   
\node[above, font=\tiny] at (0,27.5) {$b$};

\draw[->, thick] (0,20) -- (0,28);

\node[above, font=\tiny] at (20,25.5) {649376};

x-axis markings
\foreach \x in {40}
{
    \draw[thick] (\x,20) -- (\x,19.5);
}

 \node[below, font=\tiny] at (39,20.4) {40};
 \node[right, font=\tiny] at (40.7,20) {$k$};

\draw[fill=blue!30]
    (0,10) rectangle (17,11.5);

\draw[fill=blue!30]
    (17,10) rectangle (18,17);

\draw[fill=green!30]
    (18,10) rectangle (40,16);

\node[above, font=\tiny] at (8.5,10.7) {101};
\node[above, font=\tiny] at (17.5,16.2) {649477};
\node[above, font=\tiny] at (28.5,15.5) {649376};

\draw[->, thick] (0,10) -- (42,10);
    
\draw[->, thick] (0,10) -- (0,17);
    
\node[above, font=\tiny] at (0,16.5) {$b$};

x-axis markings
\foreach \x in {17,40}
{
    \draw[thick] (\x,10) -- (\x,9.5);
}

\node[below, font=\tiny] at (39,10.4) {40};
\node[right, font=\tiny] at (40.7,10) {$k$};

\draw[fill=red!30]
    (0,0) rectangle (17,1.5);
\draw[fill=red!30]
    (17,0) rectangle (18,4.5);
\draw[fill=red!30]
    (18,0) rectangle (19,1.5);
\draw[fill=red!30]
    (19,0) rectangle (20,4.5);
\draw[fill=red!30]
    (20,0) rectangle (22,1.5);
\draw[fill=red!30]
    (22,0) rectangle (24,4.5);
\draw[fill=red!30]
    (24,0) rectangle (26,1.5);

\draw[fill=red!30]
    (26,0) rectangle (27,7.2);

\draw[fill=green!30]
    (27,0) rectangle (40,6);

\node[above, font=\tiny] at (8.5,0.8) {101};
\node[above, font=\tiny] at (17,4) {1802};
\node[above, font=\tiny] at (27.5,6.3) {649578};
\node[above, font=\tiny] at (33.5,5.5) {649376};

\draw[->, thick] (0,0) -- (42,0);

\draw[->, thick] (0,0) -- (0,6.8);
    
\node[above, font=\tiny] at (0,6) {$b$};

\foreach \x in {17,26,40}
{
    \draw[thick] (\x,0) -- (\x,-0.5);
}

\node[below, font=\tiny] at (0,0) {0};
\node[below, font=\tiny] at (16,0) {17};
\node[below, font=\tiny] at (25,0) {26};
\node[below, font=\tiny] at (39,0) {40};
\node[right, font=\tiny] at (40.7,0) {$k$};

\end{tikzpicture}%
\label{subfig:comminication-efficiency-of-AdFSGOR-linear-regresssion}}
\hfil
\subfloat[]{%
\begin{tikzpicture}[x=0.095cm,y=0.115cm]

\draw[fill=green!30]
    (0,20) rectangle (40,24);
\draw[->, thick] (0,20) -- (42,20);
    
\node[above, font=\tiny] at (0,25.5) {$b$};

\draw[->, thick] (0,20) -- (0,26);

\node[above, font=\tiny] at (20,23.5) {38048};

x-axis markings
\foreach \x in {40}
{
    \draw[thick] (\x,20) -- (\x,19.5);
}

 \node[below, font=\tiny] at (39,20) {40};
 \node[right, font=\tiny] at (40.7,20) {$k$};

\draw[fill=blue!30]
    (0,10) rectangle (4,11);

\draw[fill=blue!30]
    (4,10) rectangle (5,15);

\draw[fill=green!30]
    (5,10) rectangle (40,14);

\node[above, font=\tiny] at (2,10.25) {11};
\node[above, font=\tiny] at (4.5,14.25) {38059};
\node[above, font=\tiny] at (22.5,13.5) {38048};

\draw[->, thick] (0,10) -- (42,10);
    
\draw[->, thick] (0,10) -- (0,16);
    
\node[above, font=\tiny] at (0,15.5) {$b$};

x-axis markings
\foreach \x in {4,40}
{
    \draw[thick] (\x,10) -- (\x,9.5);
}

\node[below, font=\tiny] at (4,10) {4};
\node[below, font=\tiny] at (39,10) {40};
\node[right, font=\tiny] at (40.7,10) {$k$};

\draw[fill=red!30]
    (0,0) rectangle (4,1);

\draw[fill=red!30]
    (4,0) rectangle (5,3);

\draw[fill=red!30]
    (5,0) rectangle (6,1);

\draw[fill=red!30]
    (6,0) rectangle (7,3);

\draw[fill=red!30]
    (7,0) rectangle (8,1);

\draw[fill=red!30]
    (8,0) rectangle (9,3);

\draw[fill=red!30]
    (9,0) rectangle (10,1);

\draw[fill=red!30]
    (10,0) rectangle (11,3);
\draw[fill=red!30]
    (11,0) rectangle (12,1);
\draw[fill=red!30]
    (12,0) rectangle (13,3);
\draw[fill=red!30]
    (13,0) rectangle (14,1);
\draw[fill=red!30]
    (14,0) rectangle (15,3);
\draw[fill=red!30]
    (15,0) rectangle (16,1);
\draw[fill=red!30]
    (16,0) rectangle (17,3);
\draw[fill=red!30]
    (17,0) rectangle (18,1);
\draw[fill=red!30]
    (18,0) rectangle (19,3);
\draw[fill=red!30]
    (19,0) rectangle (20,1);
\draw[fill=red!30]
    (20,0) rectangle (21,3);

\draw[fill=red!30]
    (21,0) rectangle (22,6);

\draw[fill=green!30]
    (22,0) rectangle (40,4);

\node[above, font=\tiny] at (2,0.25) {11};
\node[above, font=\tiny] at (4.5,2.5) {162};
\node[above, font=\tiny] at (21.5,5.25) {38070};
\node[above, font=\tiny] at (31,3.5) {38048};

\draw[->, thick] (0,0) -- (42,0);

\draw[->, thick] (0,0) -- (0,5.8);
   
\node[above, font=\tiny] at (0,5.5) {$b$};

\foreach \x in {4,21,40}
{
    \draw[thick] (\x,0) -- (\x,-0.5);
}

\node[below, font=\tiny] at (0,0) {0};
\node[below, font=\tiny] at (4,0) {4};
\node[below, font=\tiny] at (20,0) {21};
\node[below, font=\tiny] at (39,0) {40};
\node[right, font=\tiny] at (40.7,0) {$k$};

\end{tikzpicture}%
\label{subfig:comminication-efficiency-of-AdFSGOR-logistic-regresssion}}
\hfil
\subfloat[]{\includegraphics[width=0.5\linewidth]%
{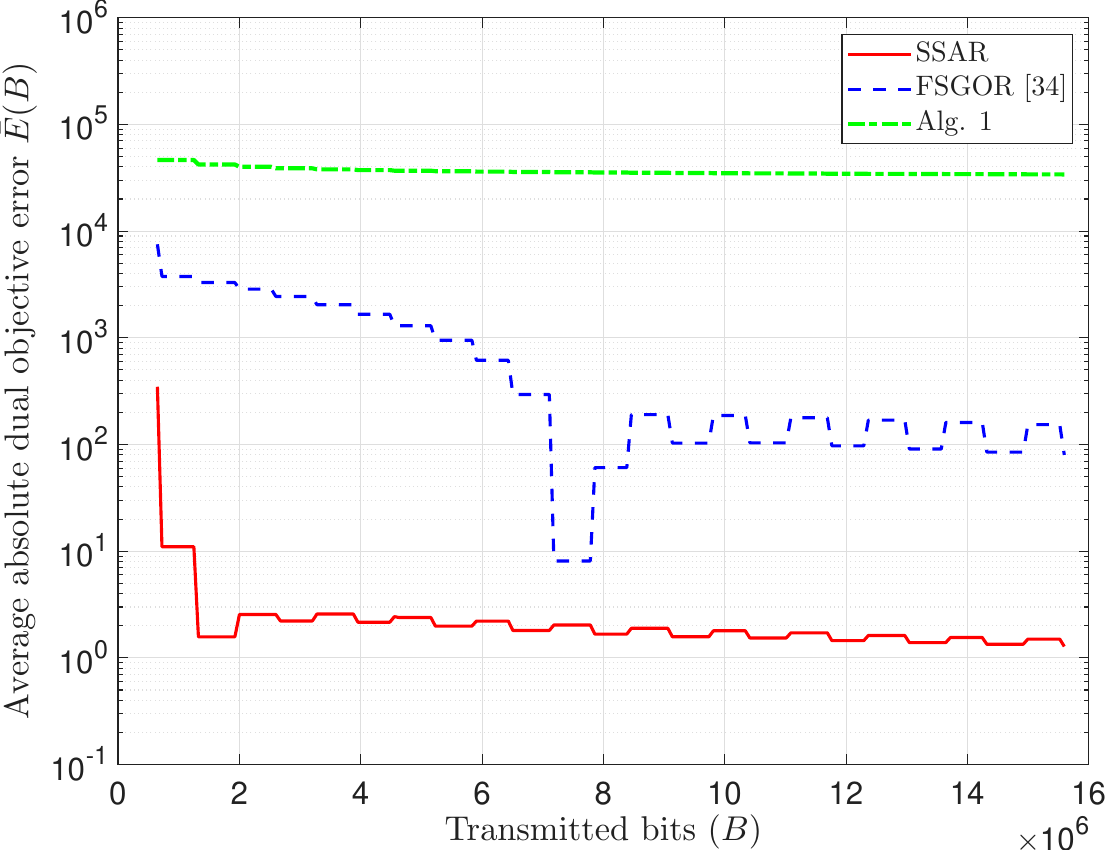}%
\label{subfig:error-vs-bits-linear-regresssion}}
\hfil
\subfloat[]{\includegraphics[width=0.5\linewidth]%
{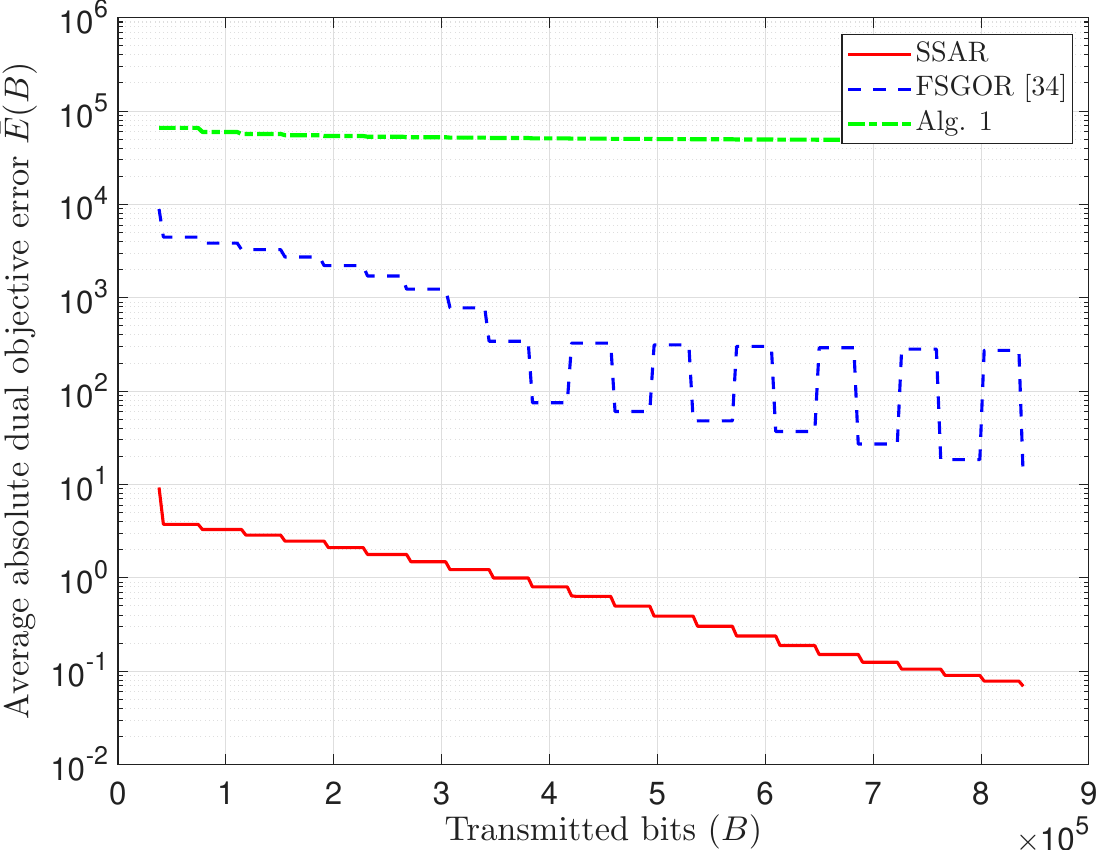}%
\label{subfig:error-vs-bits-logistic-regresssion}}
\hfil
\caption{Number of bits transmitted at each iteration [\protect\subref{subfig:comminication-efficiency-of-AdFSGOR-linear-regresssion}
and \protect\subref{subfig:comminication-efficiency-of-AdFSGOR-logistic-regresssion}]
and average absolute dual objective error versus transmitted bits
[\protect\subref{subfig:error-vs-bits-linear-regresssion}
and \protect\subref{subfig:error-vs-bits-logistic-regresssion}].
Subfigures \protect\subref{subfig:comminication-efficiency-of-AdFSGOR-linear-regresssion}
and \protect\subref{subfig:error-vs-bits-linear-regresssion}
correspond to LSR, while
\protect\subref{subfig:comminication-efficiency-of-AdFSGOR-logistic-regresssion}
and \protect\subref{subfig:error-vs-bits-logistic-regresssion}
correspond to logistic regression. Numbers above the rectangles in
\protect\subref{subfig:comminication-efficiency-of-AdFSGOR-linear-regresssion}
and \protect\subref{subfig:comminication-efficiency-of-AdFSGOR-logistic-regresssion}
give their heights. }  
\label{Fig:comminication-efficiency-of-AdFSGOR}
\end{figure}

\section{Conclusion}\label{sec:conclusion}
In this paper, an efficient subgradient-based framework, referred to as the \emph{sequential subgradient-aligned rays} (SSAR) method, was proposed to improve the convergence performance of the dual subgradient method. The proposed SSAR method was further leveraged to improve the convergence performance and communication efficiency of the standard dual decomposition method for minimizing the average of local loss functions associated with multiple agents that share a common decision variable, a widely used formulation in large-scale distributed optimization applications. Numerical results on least squares regression and logistic regression problems demonstrated that SSAR can significantly improve the convergence performance and communication efficiency of the standard dual decomposition method and outperform the recently proposed FSGOR method~\cite{abeynanda2026characteristics} in both aspects.

\begin{appendices}
\section{Proof of Lemma~\ref{lemma:feasibility-of-relaxed-problem-implies-feasibility-of-original-problem}}\label{Appendix:proof-feasibility-of-relaxed-problem-implies-feasibility-of-original-problem}
\begin{proof}
Since $(\boldsymbol{\tilde\mu},\boldsymbol{\tilde\eta})$ is a feasible
solution of~\eqref{eq:RFSGOR-feasibility-problem-relaxed}, constraints
$1$, $2$, and $3$ of~\eqref{eq:RFSGOR-feasibility-problem-relaxed} are
satisfied. Therefore, constraints $1$, $3$, and $5$ of
\eqref{eq:RFSGOR-feasibility-problem-original} are satisfied by
$\boldsymbol{\hat\mu}=\boldsymbol{\tilde\mu}$ and
$\boldsymbol{\hat\eta}=\boldsymbol{\tilde\eta}$. It remains to show that
$\boldsymbol{\hat\lambda}=c\boldsymbol{\tilde\mu}$ and
$\boldsymbol{\hat\nu}=c\boldsymbol{\tilde\eta}$ satisfy the remaining
constraints of~\eqref{eq:RFSGOR-feasibility-problem-original}. Since
$\boldsymbol{\tilde\eta}\in\operatorname{int}N_{\mathcal{X}}(\bar{\mathbf{x}})$,
there exists $\epsilon>0$ such that $\mathcal{B}(\boldsymbol{\tilde\eta},\epsilon)
\subseteq N_{\mathcal{X}}(\bar{\mathbf{x}})$,
where $\mathcal{B}(\boldsymbol{\eta}_0,r)$ denotes the ball with center
$\boldsymbol{\eta}_0$ and radius $r$. Since
$N_{\mathcal{X}}(\bar{\mathbf{x}})$ is a cone, $\forall$ $c>0$,
\begin{equation}
\label{eq:proof-feasibility-of-relaxed-problem-implies-feasibility-of-original-problem-2}
\mathcal{B}(c\boldsymbol{\tilde\eta},c\epsilon)
\subseteq N_{\mathcal{X}}(\bar{\mathbf{x}}).
\end{equation}
Let $\mathbf{a}\in\partial h_0(\bar{\mathbf{x}})$. For any
$c>\|\mathbf{a}\|/\epsilon$, we have
$\|\mathbf{a}\|<c\epsilon$. Hence, from
\eqref{eq:proof-feasibility-of-relaxed-problem-implies-feasibility-of-original-problem-2}, we have $\mathcal{B}(c\boldsymbol{\tilde\eta},\|\mathbf{a}\|)
\subseteq N_{\mathcal{X}}(\bar{\mathbf{x}}).$
%
Moreover, $\|(c\boldsymbol{\tilde\eta}-\mathbf{a})
-c\boldsymbol{\tilde\eta}\|=\|\mathbf{a}\|$, and hence
$c\boldsymbol{\tilde\eta}-\mathbf{a}
\in\texttt{bnd}~\mathcal{B}(c\boldsymbol{\tilde\eta},\|\mathbf{a}\|)$.
Then, $c\boldsymbol{\tilde\eta}-\mathbf{a}
\in N_{\mathcal{X}}(\bar{\mathbf{x}})$.
Therefore, $c\boldsymbol{\tilde\eta}
\in \mathbf{a}+N_{\mathcal{X}}(\bar{\mathbf{x}})
\subseteq \partial h_0(\bar{\mathbf{x}})
+N_{\mathcal{X}}(\bar{\mathbf{x}})$.
Thus, for sufficiently large $c$, constraint $2$ of
\eqref{eq:RFSGOR-feasibility-problem-original} is satisfied by
$\boldsymbol{\hat\nu}=c\boldsymbol{\tilde\eta}$. Finally, by constraint
$4$ of \eqref{eq:RFSGOR-feasibility-problem-original}, we have
$\boldsymbol{\hat\lambda}=(\mathbf{A}\mathbf{A}\tran)^{-1}
\mathbf{A}\boldsymbol{\hat\nu}
=c(\mathbf{A}\mathbf{A}\tran)^{-1}\mathbf{A}\boldsymbol{\tilde\eta}
=c\boldsymbol{\tilde\mu}$, 
where the last equality follows from
$\boldsymbol{\tilde\eta}=\mathbf{A}\tran\boldsymbol{\tilde\mu}$,
\cf~constraint $3$ of~\eqref{eq:RFSGOR-feasibility-problem-relaxed}.
Therefore,
$(\boldsymbol{\hat\lambda},\boldsymbol{\hat\mu},
\boldsymbol{\hat\nu},\boldsymbol{\hat\eta})=
(c\boldsymbol{\tilde\mu},
\boldsymbol{\tilde\mu},
c\boldsymbol{\tilde\eta},
\boldsymbol{\tilde\eta})$ is feasible for~\eqref{eq:RFSGOR-feasibility-problem-original}.
\end{proof}

\section{Proof of Lemma~\ref{lemma:proof-affine-over-the-line-segment}}\label{Appendix:proof-affine-over-the-line-segment}
\begin{proof}
Since $\mathbf{s}\in\partial f(\mathbf{x}_1)$ and
$\mathbf{s}\in\partial f(\mathbf{x}_2)$, the subgradient inequality gives
$f(\mathbf{x}_2)\geq f(\mathbf{x}_1)+\mathbf{s}\tran(\mathbf{x}_2-\mathbf{x}_1)$
and
$f(\mathbf{x}_1)\geq f(\mathbf{x}_2)+\mathbf{s}\tran(\mathbf{x}_1-\mathbf{x}_2)$.
Hence,
$f(\mathbf{x}_2)=f(\mathbf{x}_1)+\mathbf{s}\tran(\mathbf{x}_2-\mathbf{x}_1)$.
For any $\mathbf{x}\in\mathcal{C}$, there exists $\theta\in[0,1]$ such that
$\mathbf{x}=(1-\theta)\mathbf{x}_1+\theta\mathbf{x}_2$. By convexity,
$f(\mathbf{x})\leq(1-\theta)f(\mathbf{x}_1)+\theta f(\mathbf{x}_2)
=f(\mathbf{x}_1)+\mathbf{s}\tran(\mathbf{x}-\mathbf{x}_1)$.
On the other hand, $\mathbf{s}\in\partial f(\mathbf{x}_1)$ implies
$f(\mathbf{x})\geq f(\mathbf{x}_1)+\mathbf{s}\tran(\mathbf{x}-\mathbf{x}_1)$.
Thus, $f(\mathbf{x})=f(\mathbf{x}_1)+\mathbf{s}\tran(\mathbf{x}-\mathbf{x}_1)$
for all $\mathbf{x}\in\mathcal{C}$, and hence $f$ is affine on $\mathcal{C}$.
\end{proof}

 \end{appendices}

\bibliographystyle{IEEEtran}
\bibliography{IEEEabrv,References}

\end{document}